\documentclass[12pt,a4paper]{article}

\usepackage{latexsym}
\usepackage{amsmath}
\usepackage{amssymb}
\usepackage{amsthm}
\usepackage{amscd}
\usepackage{mathrsfs}
\usepackage[all]{xy}
\usepackage{graphicx}

\newtheorem{thm}{Theorem}[section]
\newtheorem{prop}[thm]{Proposition}

\newtheorem{lem}[thm]{Lemma}

\newtheorem{rem}[thm]{Remark}

\theoremstyle{remark}

\makeatletter

\@addtoreset{equation}{section}
\makeatother

\makeatletter
\newcommand{\subsubsubsection}{\@startsection{paragraph}{4}{\z@}%
 {1.0\Cvs \@plus.5\Cdp \@minus.2\Cdp}%
 {.1\Cvs \@plus.3\Cdp}%
 {\reset@font\sffamily\normalsize}
 }
\makeatother
\DeclareMathOperator{\Fr}{Fr}
\DeclareMathOperator{\Gal}{Gal}

\DeclareMathOperator{\tr}{tr}

\DeclareMathOperator{\End}{End}
\DeclareMathOperator{\Hom}{Hom}

\DeclareMathOperator{\Spf}{Spf}

\DeclareMathOperator{\Ker}{Ker}

\DeclareMathOperator{\Ind}{Ind} 

\DeclareMathOperator{\Tr}{Tr}
\DeclareMathOperator{\Nr}{Nr}
\DeclareMathOperator{\Trd}{Trd}
\DeclareMathOperator{\Nrd}{Nrd}
\DeclareMathOperator{\Irr}{Irr}
\DeclareMathOperator{\Disc}{Disc}
\DeclareMathOperator{\Image}{Im}

\newcommand{\cf}{\textit{cf.\ }}

\begin{document}

\title
{Geometric realization of the local Langlands correspondence 
for representations of conductor three}
\author{Naoki Imai and Takahiro Tsushima} 

\date{}

\maketitle

\footnotetext{2010 \textit{Mathematics Subject Classification}. 
 Primary: 11F80; Secondary: 11G20.}

\begin{abstract}
We prove a realization of the local Langlands correspondence 
for two-dimensional representations of a Weil group 
of conductor three 
in the cohomology of Lubin--Tate curves 
by a purely local geometric method. 
\end{abstract}

\section{Introduction}
Let $K$ be a non-archimedean local field 
with a residue field $k$ of characteristic $p$. 
Let $\mathfrak{p}$ be the maximal ideal of 
the ring of integers of $K$. 
We take an algebraic closure $K^{\mathrm{ac}}$ of $K$. 
Let $K^{\mathrm{ur}}$ be the 
maximal unramified extension of $K$ 
inside $K^{\mathrm{ac}}$. 
Let $\widehat{K}^{\mathrm{ac}}$ and 
$\widehat{K}^{\mathrm{ur}}$ denote 
the completions of $K^{\mathrm{ac}}$ and 
$K^{\mathrm{ur}}$ respectively. 
For a natural number $n$, 
we write $\mathrm{LT}(\mathfrak{p}^n)$ 
for the Lubin--Tate curve with 
full level $n$ structure over $\widehat{K}^{\mathrm{ur}}$. 
We write $W_K$ for the Weil group of $K$. 
Let $D$ be the central division algebra over $K$ of 
invariant $1/2$. 
Let $\ell$ be a prime number different from $p$. 
We take an algebraic closure 
$\overline{\mathbb{Q}}_{\ell}$ of 
$\mathbb{Q}_{\ell}$. 
Then the groups $W_K$, $\textit{GL}_2 (K)$ and $D^{\times}$ 
act on 
$\varinjlim_n H^1 _{\mathrm{c}} 
 \left( \mathrm{LT}(\mathfrak{p}^n)_{\widehat{K}^{\mathrm{ac}}}, 
 \overline{\mathbb{Q}}_{\ell} \right)$, 
and these actions partially 
realize the local Langlands correspondence 
and the local Jacquet--Langlands correspondence for 
$\textit{GL}_2$ over $K$. 
This realization was proved by 
Carayol in \cite{CarNaLT} 
using global automorphic arguments. 
However there is no known proof using only 
a local geometric method. 

In this paper, we give a purely local proof of 
a realization of the local Langlands correspondence 
for two-dimensional $W_K$-representations 
of conductor three, 
using a description of a semi-stable reduction of 
a Lubin--Tate curve in \cite{ITStab3}. 
A conductor of a representation of a Weil group 
means the Artin conductor exponent of it. 
We note that 
three is the smallest integer which is 
a conductor of 
a primitive two-dimensional $W_K$-representation. 
The calculation in 
\cite{ITStab3} is done by purely local methods. 
In \cite{ITStab3}, the actions of 
$W_K$ and $D^{\times}$ on a Lubin--Tate curve 
are calculated in some sense. 
Using the calculation, 
we can study representations of $W_K$ and $D^{\times}$ 
in the cohomology of Lubin--Tate curves. 
On the other hand, 
we have already known the relation between 
representations of $\textit{GL}_2 (K)$ and $D^{\times}$ 
in the cohomology 
by \cite[Proposition 2.1]{ITStab3}, 
which is based on a realization of 
the local Jacquet--Langlands correspondence 
proved in \cite{MieGeomJL} 
by purely local methods. 
Therefore, we can study the relation between 
representations of $W_K$ and $\textit{GL}_2 (K)$ 
in the cohomology. 
This enables us to show a realization of 
the local Langlands correspondence in 
the cohomology of Lubin--Tate curves. 

In the study of a realization of 
the local Langlands correspondence for $\textit{GL}_2$, 
the most difficult and interesting case 
is the dyadic case, which means 
the case where $p=2$. 
A proof in the case where $p$ is odd 
is similar and easier. 
Therefore, we have decided to write a proof only 
in the dyadic case. 

In the dyadic case, 
the irreducible two-dimensional $W_K$-representations of 
conductor three are primitive. 
In a construction of the local 
Langlands correspondence for 
primitive representations in \cite{KutLcGl2}, 
Weil representations are used (\cf \cite[\S 12]{BHLLCGL2}). 
On the other hand, 
our descriptions of representations of 
$\textit{GL}_2 (K)$ in the cohomology of Lubin--Tate curves 
are given by cuspidal types. 
Therefore, it is a non-trivial problem 
to check that the described 
representations in the cohomology 
actually match under the local Langlands correspondence. 

We explain the contents of this paper. 
In Section \ref{LTsp}, 
we recall definitions of Lubin--Tate curves, and 
give an easy consequence of a cohomological 
result in \cite{ITStab3}. 
In Section \ref{sstred}, 
we recall a description of a semi-stable reduction of 
a Lubin--Tate curve in \cite{ITStab3}. 

In Section \ref{cohell}, 
we study the first degree etale cohomology of an elliptic curve, 
which appears in the semi-stable reduction of 
the Lubin--Tate curve as an irreducible component. 
The cohomology of this elliptic curve gives a 
primitive $W_K$-representation of conductor three. 

In Section \ref{realcor}, 
we show that a correspondence of 
explicitly described representations 
appear in the cohomology of Lubin--Tate curves. 
In Section \ref{LLCe}, 
we show that the correspondence 
obtained in Section \ref{realcor} 
is actually the local Langlands correspondence 
by calculating epsilon factors. 
In other word, we give a description of 
the local Langlands correspondence via 
cuspidal type for representation of conductor three. 
To determine the sign of an epsilon factor, 
we calculate the Artin map for 
a wildly ramified abelian extension 
with a non-trivial ramification filtration 
by reducing it to the equal characteristic case 
using results in \cite{Delp0}. 

Following a suggestion of a referee, 
we give comments on related progresses 
after this paper was written. 
A semi-stable reduction of a Lubin--Tate curve 
studied in this paper 
is generalized to higher dimensional cases 
as reductions of 
affinoids in Lubin--Tate perfectoid spaces 
in \cite{ITsimptame} and \cite{ITsimpwild}. 
Some arguments in Section \ref{LLCe} 
of this paper is generalized to 
$\mathit{GL}_n$-cases 
in \cite{ITlgsw1} 
to study an explicit description of 
the local Langlands correspondence for 
simple supercuspidal representations of $\mathit{GL}_n$. 
On the other hand, 
we use 
the Deligne--Laumon product formula in \cite{ITlgsw1}, 
which is not of purely local nature. 
In \cite{TsuLTodd}, 
a purely local proof of the non-abelian Lubin--Tate theory 
for $\mathit{GL}_2$ is given in the 
odd equal characteristic case. 
A purely local proof of the non-abelian Lubin--Tate theory 
for a primitive Galois representation is 
not yet known outside the case studied in this paper. 

\subsection*{Acknowledgment}
The authors thank Yoichi Mieda and 
Seidai Yasuda for helpful conversations. 
They are grateful to a referee for 
comments and suggestions, which help them to 
improve expositions. 

\subsection*{Notation}
In this paper, we use the following notation. 
For a field $F$, 
an algebraic closure of $F$ is denoted by 
$F^{\mathrm{ac}}$ and 
a separable closure of $F$ is denoted by 
$F^{\mathrm{sep}}$. 
For a Galois extension $E$ over $F$, let 
$\Gal (E/F)$ denote the Galois group of the extension. 
Let $K$ be a non-archimedean local field. 
Let $\mathcal{O}_K$ denote the 
ring of integers of $K$ 
and $k$ its residue field of characteristic $p>0$. 
Let $\mathfrak{p}$ be the maximal ideal of 
$\mathcal{O}_K$. 
Let $q=|k|$. 
For $\xi \in k$, let 
$\hat{\xi} \in \mu_{q-1}(K) \cup \{ 0\}$ 
denote the lift of $\xi$. 
For a finite extension $F$ of $K$, let 
$W_F$ denote the Weil group of $F$ 
and $I_F$ denote the inertia subgroup of $W_F$. 
For a finite extension $F$ of $K$ and 
a Galois extension $L$ of $F$ in $F^{\mathrm{sep}}$, 
let $W(L/F)$ denote the image of 
$W_F$ in $\Gal (L/F)$. 
Let $K^{\mathrm{ur}}$ 
denote the maximal 
unramified extension of $K$ 
in $K^{\mathrm{ac}}$. 
The completions of $K^\mathrm{ac}$ and 
$K^\mathrm{ur}$ are denoted 
by $\widehat{K}^{\mathrm{ac}}$ and 
$\widehat{K}^{\mathrm{ur}}$ respectively. 
We write 
$\mathcal{O}_{\widehat{K}^{\mathrm{ac}}}$ and 
$\mathcal{O}_{\widehat{K}^{\mathrm{ur}}}$ for 
the rings of integers of 
$\widehat{K}^{\mathrm{ac}}$ and 
$\widehat{K}^{\mathrm{ur}}$ respectively. 
For an element 
$a \in \mathcal{O}_{\widehat{K}^{\mathrm{ac}}}$, 
we write 
$\bar{a}$ for the image of $a$ 
by the reduction map 
$\mathcal{O}_{\widehat{K}^{\mathrm{ac}}} \to k^\mathrm{ac}$. 
We fix a uniformizer $\varpi$ of $K$. 
Let $v(\cdot)$ denote the valuation of 
$K^{\mathrm{ac}}$ such that $v(\varpi)=1$. 
Let $\lvert \cdot \rvert_K$ be the 
absolute value of $K$ such that 
$\lvert \varpi \rvert_K =q^{-1}$. 
For $a,b \in K^{\mathrm{ac}}$ and 
a rational number $\alpha \in 
\mathbb{Q}_{\geq 0}$, 
we write $a \equiv b \pmod{\alpha}$
if we have $v(a-b) \geq \alpha$, 
and 
$a \equiv b \pmod{\alpha+}$
if we have $v(a-b) > \alpha$. 
For a local ring $A$, the maximal ideal of $A$ 
is denoted by $m_A$. 
For an algebraic curve $X$ over $k^{\mathrm{ac}}$, 
we denote by $X^c$ 
the smooth compactification 
of the normalization of $X$. 
For an affinoid 
$\mathbf{X}$, we write 
$\overline{\mathbf{X}}$ for its reduction. 
The category of sets is denoted by 
$\mathbf{Set}$. 
For a representation $\pi$ of a group, 
the dual representation of $\pi$ is denoted by 
$\pi^*$. 
We take rational powers of $\varpi$ compatibly 
as needed. 

\section{Lubin--Tate curve}\label{LTsp}
Let $n$ be a natural number. 
We put 
\[
 K_1 (\mathfrak{p}^n) =\biggl\{ 
 \begin{pmatrix}
 a & b \\
 c & d
 \end{pmatrix} 
 \in \textit{GL}_2 (\mathcal{O}_K ) \ \bigg| \ 
 c \equiv 0 ,\ d \equiv 1 \bmod \mathfrak{p}^n \biggr\}. 
\]
In the following, we define 
the connected Lubin--Tate curve 
$\mathbf{X}_1 (\mathfrak{p}^n )$ with level $K_1 (\mathfrak{p}^n)$. 

Let $\Sigma$ 
denote a formal $\mathcal{O}_K$-module of 
dimension $1$ and height $2$ over $k^{\mathrm{ac}}$, 
which is unique up to isomorphism. 
Let $\mathcal{C}$ 
be the category of 
Noetherian complete local 
$\mathcal{O}_{\widehat{K}^{\mathrm{ur}}}$-algebras 
with residue field $k^{\mathrm{ac}}$. 
For a one-dimensional formal 
$\mathcal{O}_K$-module 
$\mathcal{F} =\Spf A[[X]]$ over 
$A \in \mathcal{C}$ and an element 
$a \in \mathcal{O}_K$, 
we write 
$[a]_{\mathcal{F}}(X) \in A[[X]]$ 
for the $a$-multiplication 
on $\mathcal{F}$. 
For a formal one-dimensional $\mathcal{O}_K$-module 
$\mathcal{F} =\Spf A[[X]]$ over $A \in \mathcal{C}$ and 
an $A$-valued point $P$ of $\mathcal{F}$, 
the corresponding element of $m_A$ 
is denoted by $x(P)$. 
We consider the functor
\[
 \mathcal{A}_1 (\mathfrak{p}^n) \colon \mathcal{C} \to \mathbf{Set}
\]
which sends $A \in \mathcal{C}$ to the set of isomorphism 
classes of triples $(\mathcal{F},\iota,P)$, 
where 
$\mathcal{F}$ is a formal $\mathcal{O}_K$-module over $A$ 
with an isomorphism 
$\iota \colon \Sigma \simeq \mathcal {F} \otimes_A k^{\mathrm{ac}}$ and 
$P$ is a $\varpi^n$-torsion point of $\mathcal{F}$ 
such that 
\[
 \prod_{a \in \mathcal{O}_K /\mathfrak{p}^n} 
 \bigl( X- x ([a]_{\mathcal{F}} (P)) \bigr) \biggm| 
 [\varpi^n ]_{\mathcal{F}}(X) 
\]
in $A[[X]]$. This functor is represented by a regular local 
ring $\mathcal{R}_1(\mathfrak{p}^n)$. 
We write $\mathfrak{X}_1(\mathfrak{p}^n)$ for 
$\Spf \mathcal{R}_1(\mathfrak{p}^n)$. 
Its generic fiber 
is denoted by $\mathbf{X}_1(\mathfrak{p}^n)$, 
which we call the connected Lubin--Tate curve 
with level $K_1 (\mathfrak{p}^n)$. 
The space  $\mathbf{X}_1(\mathfrak{p}^n)$ 
is a rigid analytic curve over ${\widehat{K}^{\mathrm{ur}}}$. 
We can define the connected Lubin--Tate curve 
$\mathbf{X}(\mathfrak{p}^n)$ with full level $n$ structure 
by changing $P$ to be 
an $\mathcal{O}_K$-module 
homomorphism
$\phi \colon (\mathcal{O}_K/\mathfrak{p}^n)^2 \to m_A$ 
such that 
\[
 \prod_{a \in (\mathcal{O}_K/\mathfrak{p}^n)^2}
 ( X-\phi(a) ) \biggm| 
 [\varpi^n ]_{\mathcal{F}}(X)
\]
in $A[[X]]$. 
For $i \in \mathbb{Z}$, 
we can define a rigid analytic curve 
$\mathbf{X}(\mathfrak{p}^n)^{(i)}$ 
over ${\widehat{K}^{\mathrm{ur}}}$ 
by changing an isomorphism $\iota$ 
in the definition of $\mathbf{X}(\mathfrak{p}^n)$ to 
a quasi-isogeny 
$\iota \colon \Sigma \to \mathcal {F} \otimes_A k^{\mathrm{ac}}$ 
of height $i$. 
We put 
\[
 \mathrm{LT}(\mathfrak{p}^n) = 
 \coprod_{i \in \mathbb{Z}} 
 \mathbf{X}(\mathfrak{p}^n)^{(i)} . 
\]

Let $D$ be the central 
division algebra over $K$ of invariant $1/2$. 
We write $\mathcal{O}_D$ 
for the ring of integers of $D$. 
For a positive integer $m$, 
let $K_m$ be the unramified extension of $K$ of degree $m$ 
and $k_m$ be the finite extension over $k$ 
of degree $m$. 
Let $\kappa \in \Gal (K_2/K)$ be 
the non-trivial element. 
The ring $\mathcal{O}_D$ 
has the following description: 
$\mathcal{O}_D = \mathcal{O}_{K_2} \oplus \varphi \mathcal{O}_{K_2}$ 
with $\varphi^2=\varpi$ and 
$a \varphi=\varphi a^{\kappa}$ for 
$a \in \mathcal{O}_{K_2}$. 
We define an action of $\mathcal{O}_D$ on 
$\Sigma$ by 
$\zeta(X)=\bar{\zeta} X$ for 
$\zeta \in \mu_{q^2 -1} (\mathcal{O}_{K_2} )$ 
and $\varphi(X)=X^q$. 
Then this gives an isomorphism 
\begin{equation}\label{eq:ODEnd}
 \mathcal{O}_D \simeq \End (\Sigma) 
\end{equation}
by \cite[Proposition 13.10]{HGEvecLT}. 
Using the isomorphism \eqref{eq:ODEnd}, 
we can define a left action of 
$\mathcal{O}_D^{\times}$ on 
$\mathbf{X}_1(\mathfrak{p}^n)$ and 
$\mathbf{X}(\mathfrak{p}^n)$. 
We can define also 
a left action of 
$D^{\times}$ on 
$\mathrm{LT}(\mathfrak{p}^n)$ 
using the isomorphism 
$D \simeq \End (\Sigma) [1/\varpi]$ 
induced by \eqref{eq:ODEnd}. 

Let $\ell$ be a prime number different from $p$. 
We take an $\ell$-adic compactly supported cohomology of 
a rigid analytic space by regarding it as an adic space 
(\cf \cite{HubCompl}). 
We take an algebraic closure 
$\overline{\mathbb{Q}}_{\ell}$ of 
$\mathbb{Q}_{\ell}$. 
We put 
\[
 H^1 _{\mathrm{LT}} = 
 \varinjlim_n H^1 _{\mathrm{c}} 
 \left( \mathrm{LT}(\mathfrak{p}^n)_{\widehat{K}^{\mathrm{ac}}}, 
 \overline{\mathbb{Q}}_{\ell} \right). 
\]
Then we can define an action of 
$\mathit{GL}_2 (K) \times W_K \times D^{\times}$ 
on $H^1 _{\mathrm{LT}}$ 
(\cf \cite[3.2, 3.3]{DatLTel}). 

We write 
$\Irr (D^{\times},\overline{\mathbb{Q}}_{\ell})$ 
for the set of 
isomorphism classes of irreducible smooth representations of 
$D^{\times}$ over 
$\overline{\mathbb{Q}}_{\ell}$, 
and 
$\Disc (\textit{GL}_2 (K),\overline{\mathbb{Q}}_{\ell})$ 
for the set of 
isomorphism classes of irreducible 
discrete series representations of 
$\textit{GL}_2 (K)$ over 
$\overline{\mathbb{Q}}_{\ell}$. 
Let 
\[
 \mathrm{JL} \colon 
\Irr (D^{\times},\overline{\mathbb{Q}}_{\ell}) 
 \to \Disc (\textit{GL}_2 (K),\overline{\mathbb{Q}}_{\ell}) 
\]
be the local Jacquet--Langlands correspondence. 
We denote by $\mathrm{LJ}$ the inverse of 
$\mathrm{JL}$. 
Let 
$\Nrd_{D/K} \colon D^{\times} \to K^{\times}$ 
be the reduced norm map, and 
$\Trd_{D/K} \colon D\to K$ 
be the reduced trace map.

\begin{prop}\label{realJL}
For a supercuspidal representation $\pi$ of $\textit{GL}_2 (K)$ 
over $\overline{\mathbb{Q}}_{\ell}$, 
we have 
\[
 \Hom_{GL_2 (K)} (H^1 _{\mathrm{LT}} ,\pi) \simeq 
 \mathrm{LJ} (\pi)^{\oplus 2}
\]
as representations of $D^{\times}$. 
\end{prop}
\begin{proof}
Let $\omega_{\pi}$ be the central character of $\pi$. 
We take $c_{\pi} \in \overline{\mathbb{Q}}_{\ell}$ such that 
$c_{\pi}^2 =\omega_{\pi} (\varpi)$. 
We define a character $\zeta_{\pi}$ of 
$\textit{GL}_2 (K)$ by 
$\zeta_{\pi} (g) =c_{\pi}^{v(\det (g))}$, and 
a character $\xi_{\pi}$ of $D^{\times}$ by 
$\xi_{\pi} (d) =c_{\pi}^{v(\Nrd_{D/K} (d))}$ for 
$d \in D^{\times}$. 
For $n \geq 0$, let 
$\mathrm{LT}(\mathfrak{p}^n) / \varpi^{\mathbb{Z}}$ 
denote the quotient of 
$\mathrm{LT}(\mathfrak{p}^n)$ under the action of 
$\varpi^{\mathbb{Z}} \subset D^{\times}$. 
We put 
\[
 H^1 _{\mathrm{LT},\varpi} = 
 \varinjlim_n H^1 _{\mathrm{c}} 
 \bigl( \left(\mathrm{LT}(\mathfrak{p}^n) /
 \varpi^{\mathbb{Z}} \right)_{\widehat{K}^{\mathrm{ac}}}, 
 \overline{\mathbb{Q}}_{\ell} \bigr) . 
\]
Then we have 
\begin{equation*}
 \Hom_{GL_2 (K)} (H^1 _{\mathrm{LT}} ,\pi) \simeq 
 \Hom_{GL_2 (K)} (H^1 _{\mathrm{LT},\varpi} , 
 \pi \otimes \zeta_{\pi} ^{-1} ) \otimes \xi_{\pi} 
\end{equation*}
as representations of $D^{\times}$ by arguments in 
\cite[3.3]{StrJLLT}. 
Hence the claim follows from 
\cite[Proposition 2.1]{ITStab3}, 
because 
$\mathrm{LJ}(\pi \otimes \zeta_{\pi} ^{-1}) \otimes \xi_{\pi} 
 = \mathrm{LJ}(\pi)$. 
\end{proof}

\section{Semi-stable reduction of $\mathbf{X}_1(\mathfrak{p}^3)$}\label{sstred}
From now on, 
we assume that 
$p=2$. 
The dual graph of a semi-stable reduction of 
$\mathbf{X}_1(\mathfrak{p}^3)$ in this case is 
the following: \\ 
\centerline{
\xygraph{
\circ ([]!{+(+.1,+.3)} {\overline{\mathbf{Y}}^{\mathrm{c}}_{1,2}})
(-[r]
\circ ([]!{+(+.1,+.3)} {\overline{\mathbf{Z}}^{\mathrm{c}}_{1,1}})
(-[r] 
 \circ ([]!{+(+.1,+.3)} {\overline{\mathbf{Y}}^{\mathrm{c}}_{2,1}}), 
 -[dl] 
 \circ([]!{+(-.35,+.1)} 
 {\overline{\mathbf{P}}^{\mathrm{c}}_{\zeta_1}})
 ( [r] 
 ([]!{+(-.57,+0)} {\cdots} ) 
 ([]!{+(-.19,+0)} {\cdots} ) 
 ([]!{+(+.19,+0)} {\cdots} )
 ([]!{+(+.57,+0)} {\cdots} ) , 
 -[dl] 
 \circ([]!{+(+.2,-.35)} 
 {\overline{\mathbf{X}}^{\mathrm{c}}_{\zeta_1,\zeta'_1}}), 
 -[d] 
 \circ 
 ([]!{+(-.5,+0)} {\cdots} )
 ([]!{+(+.4,-.35)} 
 {\overline{\mathbf{X}}^{\mathrm{c}}_{\zeta_1,\zeta'_{q-1}}}) ), 
 -[dr] 
 \circ([]!{+(+.6,+.1)} 
 {\overline{\mathbf{P}}^{\mathrm{c}}_{\zeta_{q^2 -1}}}) 
 (-[d] 
 \circ([]!{+(+.3,-.35)} 
 {\overline{\mathbf{X}}^{\mathrm{c}}_{\zeta_{q^2-1},\zeta'_1}}), 
 -[dr] 
 \circ 
 ([]!{+(-.5,+0)} {\cdots} )
 ([]!{+(+.7,-.35)} 
 {\overline{\mathbf{X}}^{\mathrm{c}}_{\zeta_{q^2-1},\zeta'_{q-1}}}) ) 
 ) 
}} 
where 
$k_2 ^{\times} = 
 \{ \zeta_1 ,\ldots ,\zeta_{q^2 -1} \}$, 
$k^{\times} = 
 \{ \zeta'_1 ,\ldots ,\zeta'_{q-1} \}$, 
$\overline{\mathbf{Y}}_{1,2}$ 
and 
$\overline{\mathbf{Y}}_{2,1}$ 
are defined by $x^q y-xy^q=1$, 
$\overline{\mathbf{Z}}^{\mathrm{c}}_{1,1}$ and 
$\overline{\mathbf{P}}^{\mathrm{c}}_{\zeta}$ 
are isomorphic to 
$\mathbb{P}^1_{k^{\mathrm{ac}}}$, and 
$\overline{\mathbf{X}}_{\zeta,\zeta'}$ 
are defined by $z^2+z=w^3$ 
(\cf \cite[Introduction]{ITStab3}). 

For a finite extension $K'$ of $K$, 
let $\mathrm{Art}_{K'} \colon  {K'}^{\times} 
\stackrel{\sim}{\to} W^{\mathrm{ab}}_{K'}$ 
be the Artin reciprocity 
map normalized so that 
the image by $\mathrm{Art}_{K'}$ 
of a uniformizer is a lift of the geometric Frobenius element. 
We define a homomorphism 
$\lvert \cdot \rvert \colon W_K \to \mathbb{Q}_{>0}$ 
by the composition 
\[
 W_K \twoheadrightarrow W^{\mathrm{ab}}_K 
 \xrightarrow{\mathrm{Art}_K^{-1}} K^{\times} 
 \xrightarrow{\lvert \cdot \rvert_K} \mathbb{Q}_{>0}. 
\]

\subsubsection*{Definition of $\mathcal{S}$, $(W_K \times D^{\times})^0$ and $r_{\sigma}$}

We put 
\[
 \mathcal{S}=k_2 ^{\times} \times k^{\times} 
\]
and 
\[
 (W_K \times D^{\times})^0 = \bigl\{ (\sigma,d) \in W_K \times D^{\times} 
 \bigm| \lvert \Nrd_{D/K} (d) \rvert_K \cdot 
 \lvert \sigma \rvert =1 \bigr\}. 
\]
Then $(W_K \times D^{\times})^0$ acts on 
$\mathbf{X}_1(\mathfrak{p}^3)_{\widehat{K}^{\mathrm{ac}}}$. 
It induces an action of $(W_K \times D^{\times})^0$ on 
$\coprod_{(\zeta,\zeta') \in \mathcal{S}} 
 \overline{\mathbf{X}}^{\mathrm{c}}_{\zeta,\zeta'}$ 
by \cite[Proposition 5.4 and Proposition 6.12]{ITStab3}. 
For $\sigma \in W_K$, let 
$r_{\sigma}$ be the integer such that 
$\lvert \sigma \rvert =q^{-r_{\sigma}}$. 
Let $\mathcal{O}_D ^{\times} \rtimes W_K$ be the 
semidirect product where 
$\sigma \in W_K$ acts on 
$\mathcal{O}_D ^{\times}$ by 
$d \mapsto \varphi^{r_{\sigma}} d \varphi^{-r_{\sigma}}$. 
Then we have the isomorphism 
\begin{equation}\label{0isom}
 \mathcal{O}_D ^{\times} \rtimes W_K \simeq 
 (W_K \times D^{\times})^0 ; \ 
 (d,\sigma) \mapsto (\sigma, d \varphi^{-r_{\sigma}} ). 
\end{equation}
By this isomorphism, 
$\mathcal{O}_D ^{\times} \rtimes W_K$ acts on 
$\coprod_{(\zeta,\zeta') \in \mathcal{S}} 
 \overline{\mathbf{X}}^{\mathrm{c}}_{\zeta,\zeta'}$. 
We will describe this action. 

\subsubsection*{Definition of $\kappa_1$, $\kappa_2$ and $f_d$}

For $d \in \mathcal{O}_D^{\times}$, we put 
\[
 \kappa_1 (d)=\bar{d}_1, \quad 
 \kappa_2 (d)=-\bar{d}_1^{-q} \bar{d}_2 , 
\] 
where 
$d=d_1 +\varphi d_2$ with 
$d_1 \in \mathcal{O}_{K_2}^{\times}$ and 
$d_2 \in \mathcal{O}_{K_2}$. 
We take $(\zeta,\zeta') \in \mathcal{S}$. 
We put 
\[
 f_d =\Tr_{k_2/\mathbb{F}_2}
 ( \zeta^{1-q}  \zeta'^{-2} \kappa_2 (d)) 
\]
for $d \in \mathcal{O}_D ^{\times}$.

\subsubsection*{Definition of $\zeta''$, $\delta$, $\theta$, $\zeta_{3,\sigma}$, $\nu_{\sigma}$ and $\mu_{\sigma}$}

We briefly recall the definition of 
$\zeta_{3,\sigma}$, $\nu_{\sigma}$ and 
$\mu_{\sigma}$ for $\sigma \in W_K$ 
from \cite[Section 6.2.2]{ITStab3}. 
Consult there for detailed discussions. 
We choose $\zeta'' \in \mu_{3(q-1)} (K^{\mathrm{ur}})$ 
such that 
\begin{equation}\label{eq:zeta''}
 \zeta''^3=\hat{\zeta}'^4. 
\end{equation}
We take 
$\delta \in K^{\mathrm{ac}}$ such that 
$\delta^4 -\delta=1/(\zeta''\varpi^{1/3})$ and 
$\delta \varpi^{1/12} \equiv \hat{\zeta}' \zeta''^{-1} 
 \pmod{0}$. 
We take $\theta \in K^{\mathrm{ac}}$ such that 
$\theta^2 -\theta =\delta^3$. 
Note that 
$v(\delta) =-1/12$, 
$v(\theta) =-1/8$ 
and 
$\delta \in K(\zeta'' \varpi^{1/3} ,\theta)$. 

Let $\sigma \in W_K$ in this paragraph. 
We put 
\[
 \zeta_{3,\sigma} =
 \frac{\sigma(\zeta''\varpi^{\frac{1}{3}})}{\zeta''\varpi^{\frac{1}{3}}} 
 \in \mu_3 (K^{\mathrm{ur}}). 
\]
We take $\nu_{\sigma} \in \mu_3(K^{\mathrm{ur}})\cup\{0\}$ such that 
$\sigma(\delta) \equiv 
 \zeta_{3,\sigma} ^{-1} (\delta +\nu_{\sigma}) \pmod{5/6}$. 
We choose $\zeta_3 \in \mu_3(K^{\mathrm{ur}})$ 
such that $\zeta_3 \neq 1$. 
Then, we can take 
$\mu_{\sigma} \in \mu_3(K^{\mathrm{ur}})\cup\{0\}$ such that 
\[
 \mu_{\sigma} \equiv 
 \sigma (\theta)-\theta +\nu_{\sigma}^2 \delta +\nu_{\sigma}^3 
 +\sigma(\zeta_3) -\zeta_3 \pmod{0+}. 
\]

\subsubsection*{Definition of $\lambda_{\sigma}$, $\lambda$ and $Q \rtimes \mathbb{Z}$}

We put 
\[
 \lambda_{\sigma} =
 \frac{\sigma(\varpi^{\frac{1}{2(q-1)}} )}{\varpi^{\frac{1}{2(q-1)}}} 
 \in \mu_{2(q-1)}(K^{\mathrm{ac}}) 
\] 
for $\sigma \in W_K$. 
We define a character 
$\lambda \colon W_K \to k^{\times}$ by 
$\lambda (\sigma)=\bar{\lambda}_{\sigma}$. 
We put 
\[ 
 Q=
 \Biggl\{
 g(\alpha, \beta, \gamma)=
 \begin{pmatrix}
 \alpha & \beta &  \gamma \\
 & \alpha^2 &  \beta^2\\
 & & \alpha
 \end{pmatrix}
 \in \textit{GL}_3 (\mathbb{F}_4) \ \Bigg| \ 
 \alpha\gamma^2
 +\alpha^2 \gamma=\beta^3 
 \Biggr\}.
\]
We note that $|Q|=24$ (\cf \cite[p.~137]{ITStab3}). 
Let $Q \rtimes \mathbb{Z}$ be the semidirect product 
where $r \in \mathbb{Z}$ acts on $Q$ by 
$g(\alpha, \beta, \gamma) \mapsto 
 g(\alpha^{q^r} , \beta^{q^r} , \gamma^{q^r} )$. 
To clarify the dependence on $q$, 
we sometimes write 
$Q \rtimes_{(q)} \mathbb{Z}$ for 
$Q \rtimes \mathbb{Z}$. 
Let $k_2 ^{\times} \rtimes \Gal(k_2 /k)$ be the semidirect product 
with the natural action of $\Gal(k_2 /k)$ on 
$k_2 ^{\times}$. 
We consider $\mathbb{F}_4$ as a subfield of 
$k_2 \subset k^{\mathrm{ac}}$. 
Let $\Fr_q$ be the $q$-th power 
Frobenius map on $k^{\mathrm{ac}}$. 
Then 
$(Q \rtimes \mathbb{Z}) \times (k_2 ^{\times} \rtimes \Gal(k_2 /k))$ 
acts on 
$\coprod_{(\zeta,\zeta') \in \mathcal{S}} 
 \overline{\mathbf{X}}^{\mathrm{c}}_{\zeta,\zeta'}$ as a scheme over $k$ 
as follows: 
An element 
\[
 \left( \left( g(\alpha, \beta, \gamma),r \right), 
 (a,\Fr_q ^b) \right) \in 
 (Q \rtimes \mathbb{Z}) \times \left( 
 k_2 ^{\times} \rtimes \Gal(k_2 /k) \right) 
\]
acts by the isomorphism 
\[
 \overline{\mathbf{X}}_{\zeta, \zeta'} \to 
 \overline{\mathbf{X}}_{a \zeta^{q^b}, \zeta'}; \ 
 (z,w) \mapsto 
 \bigl(z^{q^{-r}} +\alpha^{-1} \beta w^{q^{-r}} +\alpha^{-1} \gamma, 
 \alpha(w^{q^{-r}} +(\alpha^{-1} \beta )^2) \bigr), 
\]
where we describe a bijection on 
$k^{\mathrm{ac}}$-valued points. 
Note that 
the action of 
$(g(1, 0, 0),r) \in Q \rtimes \mathbb{Z}$ 
is induced by the action of 
$\Fr_q^r$ on the coefficients of 
$k^{\mathrm{ac}}[z,w]/(z^2+z-w^3)$.

\begin{prop}\label{descact}
The action of $\mathcal{O}_D ^{\times} \rtimes W_K$ on 
$\coprod_{(\zeta,\zeta') \in \mathcal{S}} 
 \overline{\mathbf{X}}^{\mathrm{c}}_{\zeta,\zeta'}$ 
is described as follows: 
An element $(d,1) \in \mathcal{O}_D ^{\times} \rtimes W_K$ 
induces the isomorphism 
\[
 \overline{\mathbf{X}}_{\zeta, \zeta'} \to 
 \overline{\mathbf{X}}_{\kappa_1(d) \zeta, \zeta'}; \ 
 (z,w) \mapsto 
 (z+f_d ,w ).  
\]
For $\zeta' \in k^{\times}$, 
the action of 
$W_K \subset \mathcal{O}_D ^{\times} \rtimes W_K$ 
on 
$\coprod_{\zeta \in k_2^{\times}} 
 \overline{\mathbf{X}}^{\mathrm{c}}_{\zeta,\zeta'}$ 
factors through 
\begin{align*}
 \Xi_{\zeta'} \colon W_K &\to 
 (Q \rtimes \mathbb{Z}) \times (k_2 ^{\times} \rtimes \Gal(k_2 /k)) ;\\ 
 \sigma &\mapsto 
 \Bigl( \bigl( 
 g(\bar{\zeta}_{3,\sigma}, \bar{\zeta}_{3,\sigma}^2 
 \bar{\nu}_{\sigma}^2, 
 \bar{\zeta}_{3,\sigma} \bar{\mu}_{\sigma}), r_{\sigma} \bigr), 
 (\bar{\lambda}_{\sigma} ,\Fr_q ^{-r_{\sigma}} ) \Bigr). 
\end{align*}
\end{prop}
\begin{proof}
This follows from \cite[Proposition 5.4 and Proposition 6.12]{ITStab3}. 
\end{proof}

\subsubsection*{Definition of $\Theta_{\zeta'}$}

Let 
$\Theta_{\zeta'} \colon W_K \to 
 Q \rtimes \mathbb{Z}$ 
be the composite of $\Xi_{\zeta'}$ with 
the projection to $Q \rtimes \mathbb{Z}$. 
By \cite[Proposition 6.13]{ITStab3}, 
the map $\Theta_{\zeta'}$ gives an isomorphism 
$W(K^{\mathrm{ur}} (\varpi^{1/3}, \theta)/K) \simeq Q \rtimes \mathbb{Z}$ 
and a finite extension of $K$ inside 
$K^{\mathrm{ur}} (\varpi^{1/3}, \theta)$ 
corresponds to 
a finite index subgroup of $Q \rtimes \mathbb{Z}$. 

\section{Cohomology of elliptic curve}\label{cohell}
Let $\ell$ be an odd prime number. 
We fix an algebraic closure 
$\overline{\mathbb{Q}}_{\ell}$ of 
$\mathbb{Q}_{\ell}$. 
In the sequel, 
we consider representations of groups 
over $\overline{\mathbb{Q}}_{\ell}$. 

\subsubsection*{Definition of $Q_8$, $C_4$, $Z$, $\phi$, $\tau$ and $C_3$}

We put 
\[
 Q_8 =\{ g(1,\beta,\gamma) \in Q \}, 
\]
which is a normal subgroup of $Q$ of order $8$. 
Let $C_4 \subset Q_8$ be the cyclic subgroup of order $4$ 
generated by $g(1,1,\gamma)$ for 
$\gamma \in \mathbb{F}^{\times}_{4} \backslash \{1\}$. 
Let $Z \subset C_4$ 
be the subgroup consisting of 
$g(1,0,\gamma)$ with $\gamma^2+\gamma=0$, 
which is the center of $Q$. 
We take a faithful character 
$\phi$ of $C_4$. 
By \cite[22.2 Lemma]{BHLLCGL2}, 
there exists 
a unique irreducible 
two-dimensional representation 
$\tau$ of $Q$ 
such that 
\begin{equation}\label{qq}
 \tau |_{Z} \simeq (\phi |_{Z} )^{\oplus 2}, 
 \quad 
 \tr (g(\alpha,0,0) ; \tau ) =-1 
\end{equation}
for 
$\alpha \in \mathbb{F}^{\times}_{4} \backslash \{1\}$. 
Let $C_3 \subset Q$ be the cyclic subgroup of order $3$ 
consisting of 
$g(\alpha,0,0)$ with 
$\alpha \in \mathbb{F}^{\times}_{4}$. 
Then we have 
\begin{equation}\label{tau}
 \tau =\Ind_{C_4}^Q \phi - 
 \Ind_{Z \times C_3}^Q (\phi |_Z \otimes 1_{C_3} ) 
\end{equation}
by \cite[16.4 Lemma 2.(4)]{BHLLCGL2} and a proof of 
\cite[22.2 Lemma]{BHLLCGL2}. 

\subsubsection*{Definition of $\mathcal{E}$, $\tau_q$ and $f$}

Let $\mathcal{E}$ be the elliptic curve over $\mathbb{F}_2$ 
defined by $z^2+z=w^3$. 
Then 
$(g(\alpha, \beta, \gamma),r) \in Q \rtimes \mathbb{Z}$ 
acts on $\mathcal{E}_{k^{\mathrm{ac}}}$ by 
\[
 (z,w) \mapsto 
 \bigl(z^{q^{-r}} +\alpha^{-1} \beta w^{q^{-r}} +\alpha^{-1} \gamma, 
 \alpha(w^{q^{-r}} +(\alpha^{-1} \beta )^2)\bigr). 
\]
The action of $Q \rtimes \mathbb{Z}$ 
gives a representation 
$H^1(\mathcal{E}_{k^{\mathrm{ac}}},\overline{\mathbb{Q}}_{\ell})$ 
of $Q \rtimes \mathbb{Z}$ by the pullback by the inverse. 
For a representation $V$ of 
$Q \rtimes \mathbb{Z}$ and an integer $m$, 
we write $V(m)$ for 
the twist of $V$ by 
the character 
$Q \rtimes \mathbb{Z} \ni (g,n) \mapsto q^{-mn}$. 

We write $\tau_q$ 
for the representation 
$H^1(\mathcal{E}_{k^{\mathrm{ac}}},\overline{\mathbb{Q}}_{\ell})(1)$ 
of $Q \rtimes \mathbb{Z}$. 
Let $f$ be the degree of the extension $k$ over $\mathbb{F}_2$. 

\subsubsection*{Definition of $\eta_2$, $C$ and $\phi'$}

We are going to define 
a subgroup $C \subset Q_8 \rtimes \mathbb{Z}$ 
and a character $\phi'$ of $C$. 
To clarify the dependence on $q$, 
we sometimes write 
$C_{(q)}$ and $\phi'_{(q)}$ for 
$C$ and $\phi'$ respectively. 

First, we consider the case where $f=1$. 
Let 
\[
 C_{(2)} \subset Q_8 \rtimes_{(2)} \mathbb{Z} 
\]
be 
the subgroup which consists of 
$(g(1,\beta,\gamma),n)$ satisfying 
$g(1,\beta,\gamma) \in C_4$ if $n$ is even, and 
$g(1,\beta,\gamma) \notin C_4$ if $n$ is odd. 
We note that the index of $C_{(2)}$ in 
$Q_8 \rtimes_{(2)} \mathbb{Z}$ is two. 
We take 
$\eta_2 \in \overline{\mathbb{Q}}_{\ell}$ 
such that $\eta_2^2 -2\eta_2 +2 =0$ 
and $\phi (g(1,1,\bar{\zeta}_3))=-\eta_2^2/2$. 
We note that $\eta_2^4 =-4$. 
We define a character 
\[
 \phi'_{(2)} \colon C_{(2)} \to \overline{\mathbb{Q}}_{\ell}^{\times} 
\]
by sending 
$(g(1,\bar{\zeta}_3,\bar{\zeta}_3),1)$ to $\eta_2 /2$ 
and 
$(g(1,0,0),2)$ to $(-1)/2$. 
We note that 
$(g(1,\bar{\zeta}_3,\bar{\zeta}_3),1)$ 
and 
$(g(1,0,0),2)$ 
generates $C_{(2)}$ as a group. 

In general, 
let $C_{(q)}$ be the inverse image of 
$C_{(2)}$ under the group homomorphism 
\[
 Q_8 \rtimes_{(q)} \mathbb{Z} \to Q_8 \rtimes_{(2)} \mathbb{Z}; \ 
 (g,n) \mapsto (g,fn). 
\]
Let $\phi'_{(q)}$ be the character of $C_{(q)}$ 
induced by $\phi'_{(2)}$ and the homomorphism 
$C_{(q)} \to C_{(2)}$. 
We have $\phi'_{(q)}|_{C_4} =\phi$ 
by the construction. 
We note that $C_{(q)}=C_4 \times \mathbb{Z}$ 
and 
$\phi'_{(q)} (g,n) =\phi (g) (-2)^{(-fn)/2}$, 
if $f$ is even. 

\begin{lem}\label{quadind}
We have an isomorphism 
$\tau_q |_{Q_8 \rtimes \mathbb{Z}} \simeq 
 \Ind _C ^{Q_8 \rtimes \mathbb{Z}} \phi'$. 
\end{lem}
\begin{proof}
It suffices to consider the case where $f=1$, 
because  
the claimed isomorphisms for general cases 
are induced from 
the isomorphism for this case by 
the group homomorphism 
\[
 Q_8 \rtimes_{(q)} \mathbb{Z} \to Q_8 \rtimes_{(2)} \mathbb{Z}; \ 
 (g,n) \mapsto (g,fn). 
\]

We assume that $f=1$. 
We know that 
$\tau_2 |_{Q_8} \simeq \Ind_{C_4} ^{Q_8} \phi$,  
and these representations are irreducible. 
Hence the $\overline{\mathbb{Q}}_{\ell}$-vector space 
\[
 \Hom_{Q_8} \left( \tau_2 |_{Q_8}, 
 (\Ind _C ^{Q_8 \rtimes \mathbb{Z}} \phi') |_{Q_8} \right) 
\]
is $1$-dimensional, and  
$Q_8 \rtimes \mathbb{Z}$ acts 
on the $1$-dimensional subspace 
by a character $\chi$, 
which factors through 
the projection $Q_8 \rtimes \mathbb{Z} \to \mathbb{Z}$. 
Then we have 
\[
 \tau_2 |_{Q_8 \rtimes \mathbb{Z}} \simeq 
 (\Ind _C ^{Q_8 \rtimes \mathbb{Z}} \phi') \otimes \chi.
\]
Therefore, it suffices to show that 
\begin{equation}\label{trtau2Ind}
 \tr \bigl( (g(1,\bar{\zeta}_3,\bar{\zeta}_3), 1) ;\tau_2 \bigr) 
 = 
 \tr \bigl( (g(1,\bar{\zeta}_3,\bar{\zeta}_3), 1) ; 
 \Ind _C ^{Q_8 \rtimes \mathbb{Z}} \phi' \bigr) 
 \neq 0, 
\end{equation}
since \eqref{trtau2Ind} implies that $\chi$ is trivial. 
We put 
\[
 \phi''\bigl( (g,n) \bigr)= 
 \phi' \bigl( (g(1,0,0), 1) (g, n) (g(1,0,0), -1) \bigr)
\]
for $(g,n) \in Q_8 \rtimes \mathbb{Z}$. 
Then we have 
\begin{equation}\label{trInd1}
\begin{split}
 \tr \bigl( (g(1,\bar{\zeta}_3,\bar{\zeta}_3), 1) ; 
 \Ind _C ^{Q_8 \rtimes \mathbb{Z}} \phi' \bigr) 
 &= 
 \phi' \bigl( (g(1,\bar{\zeta}_3,\bar{\zeta}_3), 1) \bigr) + 
 \phi'' \bigl( 
 (g(1,\bar{\zeta}_3,\bar{\zeta}_3), 1) \bigr) \\ 
 &=
 \frac{\eta_2}{2} -\frac{\eta_2^3}{4} =1 , 
\end{split}
\end{equation}
where we use 
\[
 (g(1,0,0), 1) (g(1,\bar{\zeta}_3,\bar{\zeta}_3), 1) (g(1,0,0), -1) = 
 (g(1,\bar{\zeta}_3,\bar{\zeta}_3), 1)^3 (g(1,0,0), -2) 
\]
at the second equality. 

Let $\Fr_{2,\mathcal{E}}$ be the absolute $2$-th power 
Frobenius endomorphism of $\mathcal{E}_{k^{\mathrm{ac}}}$. 
By the Lefschetz trace formula, we have 
\begin{align*}\label{LTFe}
 2+1 -&\tr \bigl( (g(1,\bar{\zeta}_3 ,\bar{\zeta}_3), 1); 
 H^1(\mathcal{E}_{k^{\mathrm{ac}}},\overline{\mathbb{Q}}_{\ell}) \bigr) = 
 \bigl\lvert \{ P \in \mathcal{E}(k^{\mathrm{ac}}) \mid 
 \bigl((g(1,\bar{\zeta}_3,\bar{\zeta}_3),1)^{-1} \circ \Fr_{2,\mathcal{E}} \bigr) 
 P=P \} \bigr\rvert \\ 
 &= \bigl\lvert \{ (z,w) \in k^{\mathrm{ac}} \times k^{\mathrm{ac}} \mid 
 z^2+z=w^3,\ z=z^2+\bar{\zeta}_3^2 w^2 +\bar{\zeta}_3,\ 
 w=w^2+\bar{\zeta}_3 \} \bigr\rvert +1 \\ 
 &= \bigl\lvert \{ (z,w) \in k^{\mathrm{ac}} \times k^{\mathrm{ac}} \mid 
 z^2+z=w^3,\ w^3 +\bar{\zeta}_3^2 w^2 +\bar{\zeta}_3 =0,\ 
 w^2+w+\bar{\zeta}_3 =0 \} \bigr\rvert +1 
 =1. 
\end{align*}
Hence, we have 
\begin{equation}\label{trtau21}
 \tr \bigl( (g(1,\bar{\zeta}_3,\bar{\zeta}_3), 1) ;\tau_2 \bigr) =1. 
\end{equation}
The claim \eqref{trtau2Ind} follows from 
\eqref{trInd1} and \eqref{trtau21}. 
\end{proof}

For any positive integer $m$, 
let 
\[
 \mathrm{fr}_{2^m} \colon \mathcal{E}_{k^{\mathrm{ac}}} 
 \to \mathcal{E}_{k^{\mathrm{ac}}} 
\] 
be the base change to $k^{\mathrm{ac}}$ 
of the $2^m$-th power 
absolute Frobenius endomorphism of 
$\mathcal{E}$. 

\begin{lem}\label{Trtau2}
We assume that $f=1$. 
Then we have 
\[
 \tr \bigl( (g(1,0,0),n);\tau_2 \bigr) =
 \begin{cases}
 0 & \textrm{if $n=1$,} \\ 
 -1 & \textrm{if $n=2$} 
 \end{cases}
\] 
for 
$(g(1,0,0),n) \in Q \rtimes \mathbb{Z}$. 
\end{lem}
\begin{proof}
We have 
\[
 \tr \bigl( (g(1,0,0),n); 
 H^1(\mathcal{E}_{k^{\mathrm{ac}}},\overline{\mathbb{Q}}_{\ell}) \bigr) =
 \tr \bigl( \mathrm{fr}_{2^n}^* ; 
 H^1(\mathcal{E}_{k^{\mathrm{ac}}},\overline{\mathbb{Q}}_{\ell}) \bigr) =
 \begin{cases}
 0 & \textrm{if $n=1$,} \\ 
 -4 & \textrm{if $n=2$}, 
 \end{cases}
\] 
where the last equality follows from 
$\lvert \mathcal{E}(\mathbb{F}_2) \rvert =3$, 
$\lvert \mathcal{E}(\mathbb{F}_4) \rvert =9$ 
and the Lefschetz trace formula. 
The claim follows from this. 
\end{proof}

\begin{lem}\label{dettau}
We have 
$\det \bigl( (g,n);\tau_q \bigr) =q^{-n}$ for 
$(g,n) \in Q \rtimes \mathbb{Z}$. 
\end{lem}
\begin{proof}
We have an isomorphism 
$\tau_q |_Q \simeq \tau$ 
as $Q$-representations 
by \cite[Lemma 7.7]{ITStab3}. 

First we are going to show that 
$\det \tau =1$. 
We see that 
$\det \tau$ factors through $Q/Q_8$, 
because $Q/Q_8$ is the maximal abelian quotient of $Q$. 
By \eqref{tau}, we know that 
$\tau$ is self-dual. 
Hence, the character of $Q/Q_8$ 
induced from $\det \tau$ is trivial, 
since $\lvert Q/Q_8 \rvert =3$. 
Therefore, we have 
$\det \tau =1$. 

Since $\tau_q$ is induced from 
$\tau_2$ by 
the group homomorphism 
\[
 Q \rtimes_{(q)} \mathbb{Z} \to Q \rtimes_{(2)} \mathbb{Z}; \ 
 (g,n) \mapsto (g,fn), 
\]
it suffices to 
show the claim 
in the case where $f=1$. 

We assume that $f=1$. 
Let $\omega_1$ and $\omega_2$ be the non-trivial characters 
of $C_3$. 
Then, we have a direct decomposition 
\[
 \tau_2|_{C_3} \cong \omega_1 \oplus \omega_2 
\]
by \eqref{qq}. 
We fix a basis in the above decomposition. 
Then the action of 
$(g(1,0,0),1) \in Q \rtimes \mathbb{Z}$ 
can be written as 
\[
 \begin{pmatrix}
 0 & a \\ 
 b & 0 
 \end{pmatrix} 
\]
for some $a, b \in \overline{\mathbb{Q}}_{\ell}^{\times}$, 
since we have 
$(g(1,0,0),1) c = c^2 (g(1,0,0),1)$ in $Q \rtimes \mathbb{Z}$ 
for $c \in C_3$. 
By Lemma \ref{Trtau2}, 
we have $2ab=-1$. 
Hence, we have 
\[
 \det \bigl( (g(1,0,0),1);\tau_2 \bigr) =-ab=2^{-1}. 
\]
The claim follows from this and $\det \tau =1$. 
\end{proof}

\subsubsection*{Definition of $\tau_{\zeta'}$}

Let $\tau_{\zeta'}$ be the 
representation of $W_K$ induced from 
the $(Q \rtimes \mathbb{Z})$-representation 
$\tau_q$ by $\Theta_{\zeta'}$. 
We say that 
a continuous two-dimensional representation $V$ of $W_K$ 
over $\overline{\mathbb{Q}}_{\ell}$ is primitive, 
if there is no pair of a quadratic extension $K'$ and 
a continuous character $\chi$ of $W_{K'}$ such that 
$V \simeq \Ind_{W_{K'}}^{W_K} \chi$. 
The representation 
$\tau_{\zeta'}$ is primitive of conductor $3$ 
by \cite[Lemma 7.8]{ITStab3}.

\section{Realization of correspondence}\label{realcor}

\subsubsection*{Definition of $\chi_2$, $L$ and $\psi_K$}

Let 
$\chi_2 \colon \mathbb{F}_2 \to \overline{\mathbb{Q}}_{\ell}^{\times}$ 
be the non-trivial character. 
We put 
\begin{equation}\label{JPdef}
 \mathfrak{J} =\biggl\{ 
 \begin{pmatrix}
 a & b \\
 c & d
 \end{pmatrix} 
 \in M_2 (\mathcal{O}_K ) \ \bigg| \ 
 c \equiv 0 \bmod \mathfrak{p} \biggr\}, \quad 
 \mathfrak{P} =\biggl\{ 
 \begin{pmatrix}
 a & b \\
 c & d
 \end{pmatrix} 
 \in \mathfrak{J} \ \bigg| \ 
 a \equiv d \equiv 0 \bmod \mathfrak{p} \biggr\}, 
\end{equation}
and 
$U_{\mathfrak{J}}^1 =1+ \mathfrak{P} \subset M_2 (\mathcal{O}_K )$. 
We put $L=K(\varphi) \subset D$, and consider 
$L$ as a $K$-subalgebra of $M_2 (K)$ by the embedding 
\begin{equation}\label{embL}
 L \hookrightarrow M_2 (K) ; \ \varphi \mapsto 
 \begin{pmatrix}
 0 & 1 \\
 \varpi & 0
 \end{pmatrix}. 
\end{equation}
We put $U_D ^1 =1+\varphi \mathcal{O}_D$. 
We take an additive character 
$\psi_K \colon K \to \overline{\mathbb{Q}}_{\ell}^{\times}$ 
such that 
$\psi_K (a) =(\chi_2 \circ 
 \Tr_{k/\mathbb{F}_2} )(\bar{a})$ 
for $a \in \mathcal{O}_K$. 

For a finite abelian group $A$, 
the character group 
$\Hom_{\mathbb{Z}} (A,\overline{\mathbb{Q}}_{\ell} ^{\times})$ 
is denoted by $A^{\vee}$. 

\subsubsection*{Definition of $\Lambda_{\zeta',\chi,c}$ and $\pi_{\zeta',\chi,c}$}

Let $\zeta' \in k^{\times}$, 
$\chi \in (k^{\times})^{\vee}$ and 
$c \in \overline{\mathbb{Q}}_{\ell}^{\times}$. 
We define a character 
$\Lambda_{\zeta',\chi,c} \colon L^{\times} U_{\mathfrak{J}}^1 
 \to \overline{\mathbb{Q}}_{\ell}^{\times}$ by 
\begin{align*}
 & \Lambda_{\zeta',\chi,c} (\varphi)=-c,\\ 
 & \Lambda_{\zeta',\chi,c} (a) =\chi (\bar{a}) \ 
 \textrm{for} \ 
 a \in \mathcal{O}_L ^{\times}, \\ 
 & \Lambda_{\zeta',\chi,c} (x)=(\psi_K \circ 
 \tr )\left( \hat{\zeta}'^{-2} \varphi^{-1}(x-1) \right) 
 \ 
 \textrm{for} \ 
 x \in U_{\mathfrak{J}}^1. 
\end{align*}
We put 
\[
 \pi_{\zeta',\chi,c} = 
 \mathrm{c\mathchar`-Ind}_{L^{\times} U_{\mathfrak{J}}^1}^{\textit{GL}_2(K)} 
 \Lambda_{\zeta',\chi,c}. 
\]
By \cite[Proposition 1.3]{ITsimpJL}, 
$\pi_{\zeta',\chi,c}$ is a supercuspidal representation of 
conductor $3$, 
and 
any supercuspidal representation of conductor $3$ 
is isomorphic to 
$\pi_{\zeta',\chi,c}$ for some 
$\zeta' \in k^{\times}$, 
$\chi \in (k^{\times})^{\vee}$ and 
$c \in \overline{\mathbb{Q}}_{\ell}^{\times}$. 

\subsubsection*{Definition of $\theta_{\zeta',\chi,c}$ and $\rho_{\zeta', \chi,c}$}

Next, we define a character 
$\theta_{\zeta',\chi,c} \colon L^{\times} U_D ^1 
 \to \overline{\mathbb{Q}}_{\ell}^{\times}$ by 
\begin{align*}
 &\theta_{\zeta',\chi,c} (\varphi)=c, \\ 
 &\theta_{\zeta',\chi,c} (a) =\chi (\bar{a}) \ 
 \textrm{for} \ 
 a \in \mathcal{O}_L ^{\times}, \\ 
 & \theta_{\zeta',\chi,c} (d) =(\chi_2 \circ \Tr_{k_2 /\mathbb{F}_2}) 
 \left( {\zeta'}^{-2} \kappa_2 (d) \right) \ 
 \textrm{for} \ 
 d \in U_D ^1. 
\end{align*}
We put 
\[
 \rho_{\zeta', \chi,c} = 
 \Ind_{L^{\times} U_D ^1}^{D^{\times}} \theta_{\zeta', \chi,c}. 
\]
The representation 
$\rho_{\zeta', \chi,c}$ is irreducible 
by \cite[54.4 Proposition (1)]{BHLLCGL2}. 

\begin{prop}\label{expJL}
For $\zeta' \in k^{\times}$, 
$\chi \in (k^{\times})^{\vee}$ and 
$c \in \overline{\mathbb{Q}}_{\ell}^{\times}$, 
we have 
$\mathrm{JL} (\rho_{\zeta', \chi,c}) =\pi_{\zeta',\chi,c}$. 
\end{prop}
\begin{proof}
This follows from \cite[56.5]{BHLLCGL2}, because 
\[
 (\psi_K \circ \Trd_{D/K} )\left( 
 \hat{\zeta}'^{-2} \varphi^{-1}(d-1) \right) 
 = 
 (\chi_2 \circ \Tr_{k_2 /\mathbb{F}_2}) 
 \left( {\zeta'}^{-2} \kappa_2 (d) \right) 
\] 
for $d \in U_D ^1$. 
\end{proof}

\begin{rem}
In \cite[56.1]{BHLLCGL2}, 
the local Jacquet--Langlands correspondence for $D^{\times}$ 
is characterized by coincidences of 
$L$-functions and $\epsilon$-factors. 
However, we can check that 
the correspondence between 
$\rho_{\zeta', \chi,c}$ and $\pi_{\zeta',\chi,c}$ satisfies 
the characterization by trace identities (\cf \cite{ITsimpJL} ). 
We note that the existence of 
the local Jacquet--Langlands correspondence for $D^{\times}$ 
satisfying the characterization by trace identities 
is proved in \cite{MieGeomJL} by purely local methods. 
\end{rem}

\subsubsection*{Definition of $\phi_c$ and $\tau_{\zeta', \chi, c}$}

For 
$c \in \overline{\mathbb{Q}}_{\ell}^{\times}$,
let $\phi_c \colon W_K \to \overline{\mathbb{Q}}_{\ell}^{\times}$ 
be the character defined by 
$\phi_c (\sigma)=c^{r_{\sigma}}$. 
For $\zeta' \in k^{\times}$, 
$\chi \in (k^{\times})^{\vee}$ and 
$c \in \overline{\mathbb{Q}}_{\ell}^{\times}$, 
we put 
\[
 \tau_{\zeta', \chi, c} =
 \tau_{\zeta'} \otimes (\chi \circ \lambda) \otimes \phi_c. 
\]
For a representation $V$ of 
a Weil group and an integer $m$, 
we write $V(m)$ for 
the $m$-times Tate twist of $V$. 
We choose 
$(-2)^{1/2} \in \overline{\mathbb{Q}}_{\ell}$. 

\begin{thm}\label{realLLC}
For $\zeta' \in k^{\times}$, 
$\chi \in (k^{\times})^{\vee}$ and 
$c \in \overline{\mathbb{Q}}_{\ell}^{\times}$, 
we have 
\[
 \Hom_{GL_2 (K)} (H^1 _{\mathrm{LT}} ,\pi_{\zeta', \chi,c}) \simeq 
 \tau_{\zeta', \chi, c} \otimes \rho_{\zeta', \chi,c} 
\]
as representations of $W_K \times D^{\times}$. 
\end{thm}
\begin{proof}
Let $\zeta' \in k^{\times}$, 
$\chi \in (k^{\times})^{\vee}$ and 
$c \in \overline{\mathbb{Q}}_{\ell}^{\times}$. 
By Proposition \ref{realJL} and 
Proposition \ref{expJL}, we know that 
\[
 \Hom_{D^{\times}} \bigl( \rho_{\zeta', \chi,c} , 
 \Hom_{GL_2 (K)} (H^1 _{\mathrm{LT}} ,\pi_{\zeta', \chi,c}) \bigr) 
 \simeq \tau' 
\]
for some two-dimensional $W_K$-representation $\tau'$. 
First, we will show that 
$\tau' =\tau_{\zeta', \chi, c}$. 

We put 
\[
 H^1 _{\mathbf{X}} = 
 \varinjlim_n H^1 _{\mathrm{c}} 
 \left( \mathbf{X}(\mathfrak{p}^n)_{\widehat{K}^{\mathrm{ac}}}, 
 \overline{\mathbb{Q}}_{\ell} \right)
\] 
and 
\[
 (\textit{GL}_2 (K) \times W_K \times D^{\times})^0 
 = \bigl\{ (g,\sigma,d) \in \textit{GL}_2 (K) \times W_K \times D^{\times} 
 \bigm| \lvert \det (g)^{-1} \Nrd_{D/K} (d) \rvert_K \cdot 
 \lvert \sigma \rvert =1 \bigr\}. 
\]
Then we have 
\begin{align*}
 \Hom_{GL_2 (K)} (H^1 _{\mathrm{LT}} ,\pi_{\zeta', \chi,c}) &\simeq 
 \Hom_{GL_2 (K)} \left( 
 \mathrm{c\mathchar`-Ind}
 _{(\textit{GL}_2 (K) \times W_K \times D^{\times})^0 }
 ^{\textit{GL}_2 (K) \times W_K \times D^{\times}} 
 H^1 _{\mathbf{X}} ,\pi_{\zeta', \chi,c} \right) \\ 
 &\subset 
 \Hom_{K_1(\mathfrak{p}^3)} \left( 
 \mathrm{c\mathchar`-Ind}
 _{K_1(\mathfrak{p}^3) \times (W_K \times D^{\times})^0 }
 ^{K_1(\mathfrak{p}^3) \times W_K \times D^{\times}} 
 H^1 _{\mathbf{X}} ,\pi_{\zeta', \chi,c} \right) \\ 
 &\subset 
 \Hom_{\overline{\mathbb{Q}}_{\ell}} \left( 
 \mathrm{c\mathchar`-Ind}
 _{(W_K \times D^{\times})^0 }
 ^{W_K \times D^{\times}} 
 H^1 _{\mathrm{c}} 
 \left( \mathbf{X}_1 (\mathfrak{p}^3)_{\widehat{K}^{\mathrm{ac}}}, 
 \overline{\mathbb{Q}}_{\ell} \right) , \overline{\mathbb{Q}}_{\ell} \right), 
\end{align*}
where the last inclusion follows by taking the 
$K_1(\mathfrak{p}^3)$-invariant part 
and using \cite[5 Th\'{e}or\`{e}me]{JPSSCond}, 
because the conductor of 
$\pi_{\zeta', \chi,c}$ is three. 
Hence, we obtain 
\begin{align} 
 \tau' &\simeq \Hom \bigl( \rho_{\zeta', \chi,c} , 
 \Hom_{GL_2 (K)} (H^1 _{\mathrm{LT}} ,\pi_{\zeta', \chi,c}) 
 \bigr) \notag \\ 
 &\subset 
 \Hom_{D^{\times}} \left( \rho_{\zeta', \chi,c} ,
 \left( \mathrm{c\mathchar`-Ind}
 _{(W_K \times D^{\times})^0 }
 ^{W_K \times D^{\times}} 
 H^1 _{\mathrm{c}} 
 \left(\mathbf{X}_1 (\mathfrak{p}^3)_{\widehat{K}^{\mathrm{ac}}}, 
 \overline{\mathbb{Q}}_{\ell} \right) \right)^* \right) \label{tau'incl} \\
 &\simeq 
 \left( \rho_{\zeta', \chi,c}^* \otimes 
 \left( \mathrm{c\mathchar`-Ind}
 _{(W_K \times D^{\times})^0 }
 ^{W_K \times D^{\times}} 
 H^1 _{\mathrm{c}} 
 \left( \mathbf{X}_1 (\mathfrak{p}^3)_{\widehat{K}^{\mathrm{ac}}}, 
 \overline{\mathbb{Q}}_{\ell} \right) \right)^* \right)^{D^{\times}} \\ 
 &\simeq 
 \Hom_{D^{\times}} \left( 
 \mathrm{c\mathchar`-Ind}
 _{(W_K \times D^{\times})^0 }
 ^{W_K \times D^{\times}} 
 H^1 _{\mathrm{c}} 
 \left( \mathbf{X}_1 (\mathfrak{p}^3)_{\widehat{K}^{\mathrm{ac}}}, 
 \overline{\mathbb{Q}}_{\ell} \right) , 
 \rho_{\zeta', \chi,c} ^* \right) \notag \\ 
 &\simeq 
 \Hom_{D^{\times}} \left( 
 \mathrm{c\mathchar`-Ind}
 _{(W_K \times D^{\times})^0 }
 ^{W_K \times D^{\times}} 
 \left( \bigoplus_{\zeta \in k_2 ^{\times}} H^1 
 (\overline{\mathbf{X}}^{\mathrm{c}}_{\zeta,\zeta'} , 
 \overline{\mathbb{Q}}_{\ell} )^* (-1) \right) , 
 \rho_{\zeta', \chi,c} ^* \right), \notag 
\end{align}
where the last isomorphism follows from 
\cite[Proposition 7.3, Proposition 7.9 and Theorem 7.16]{ITStab3} 
by studying only 
$\mathcal{O}_D ^{\times}$-actions. 
We remark that 
\cite[Theorem 7.16]{ITStab3} is based on 
\cite[Theorem 5.3]{ITCohrigc} where we use Berkovich spaces, 
but it does not matter for Lubin--Tate spaces 
by \cite[Lemma 4.4.6]{FarCohpdiv}. 
As vector spaces, the last space is isomorphic to 
\[
 \Hom_{D^{\times}} \left( 
 \mathrm{c\mathchar`-Ind}
 _{\mathcal{O}_D ^{\times}} ^{D^{\times}} 
 \left( \bigoplus_{\zeta \in k_2 ^{\times}} H^1 
 (\overline{\mathbf{X}}^{\mathrm{c}}_{\zeta,\zeta'} , 
 \overline{\mathbb{Q}}_{\ell} )^* \right) , 
 \rho_{\zeta', \chi,c} ^* \right) 
 \simeq 
 \Hom_{\mathcal{O}_D ^{\times}} \left( 
 \bigoplus_{\zeta \in k_2 ^{\times}} H^1 
 (\overline{\mathbf{X}}^{\mathrm{c}}_{\zeta,\zeta'} , 
 \overline{\mathbb{Q}}_{\ell} )^* , 
 \rho_{\zeta', \chi,c} ^* |_{\mathcal{O}_D ^{\times}} \right) 
\] 
which is two-dimensional by 
\cite[Proposition 7.9]{ITStab3}. 
Hence, the inclusion in 
\eqref{tau'incl} is an equality. 
Therefore it suffices to show that 
there is a non-trivial homomorphism 
\[
 \mathrm{c\mathchar`-Ind}
 _{(W_K \times D^{\times})^0 }
 ^{W_K \times D^{\times}} \left( 
 \bigoplus_{\zeta \in k_2 ^{\times}} H^1 
 (\overline{\mathbf{X}}^{\mathrm{c}}_{\zeta,\zeta'} , 
 \overline{\mathbb{Q}}_{\ell} )^* (-1) \right) \longrightarrow 
 \tau_{\zeta', \chi,c}^* \otimes \rho_{\zeta', \chi,c}^*
\]
as representations of $W_K \times D^{\times}$. 
By the Frobenius reciprocity, 
this is equivalent to give a non-trivial homomorphism 
\[
 (\tau_{\zeta', \chi,c} \otimes \rho_{\zeta', \chi,c} )
 |_{(W_K \times D^{\times})^0} 
 \longrightarrow 
 \bigoplus_{\zeta \in k_2 ^{\times}} H^1 
 (\overline{\mathbf{X}}^{\mathrm{c}}_{\zeta,\zeta'} , 
 \overline{\mathbb{Q}}_{\ell} )(1) 
\]
as representations of 
$(W_K \times D^{\times})^0$. 
We put 
\[
 (\mathcal{O}_D ^{\times} \rtimes W_K)^0 = 
 \{ (d,\sigma) \in \mathcal{O}_D ^{\times} \rtimes W_K \mid 
 \kappa_1 (d) \bar{\lambda}_{\sigma} =1\} 
\]
and consider this group as a subgroup of 
$(W_K \times D^{\times})^0$ by 
the isomorphism \eqref{0isom}. 
Then we have 
\[
 \bigoplus_{\zeta \in k_2 ^{\times}} H^1 
 (\overline{\mathbf{X}}^{\mathrm{c}}_{\zeta,\zeta'} , 
 \overline{\mathbb{Q}}_{\ell} )(1) 
 \simeq 
 \Ind _{(\mathcal{O}_D ^{\times} \rtimes W_K)^0 }
 ^{(W_K \times D^{\times})^0} H^1 
 (\overline{\mathbf{X}}^{\mathrm{c}}_{1,\zeta'} , 
 \overline{\mathbb{Q}}_{\ell} )(1), 
\]
because 
the action of 
$\mathcal{O}_D ^{\times} \rtimes W_K$ on 
$\coprod_{(\zeta,\zeta') \in \mathcal{S}} 
 \overline{\mathbf{X}}^{\mathrm{c}}_{\zeta,\zeta'}$ 
permutes the connected components transitively 
and 
$(\mathcal{O}_D ^{\times} \rtimes W_K)^0$ 
is the stabilizer of 
the connected component 
$\overline{\mathbf{X}}^{\mathrm{c}}_{1,\zeta'}$ 
by Proposition \ref{descact}. 
Hence, we have 
\begin{align*}
 \Hom_{(W_K \times D^{\times})^0} & \left( 
 (\tau_{\zeta',\chi,c} \otimes \rho_{\zeta',\chi,c} ) 
 |_{(W_K \times D^{\times})^0} , 
 \bigoplus_{\zeta \in k_2 ^{\times}} H^1 
 (\overline{\mathbf{X}}^{\mathrm{c}}_{\zeta,\zeta'} , 
 \overline{\mathbb{Q}}_{\ell} )(1) \right) \\ 
 &\simeq 
 \Hom_{(\mathcal{O}_D ^{\times} \rtimes W_K)^0} \left( 
 (\tau_{\zeta', \chi,c} \otimes \rho_{\zeta', \chi,c})
 |_{(\mathcal{O}_D ^{\times} \rtimes W_K)^0} , 
 H^1 (\overline{\mathbf{X}}^{\mathrm{c}}_{1,\zeta'} , 
 \overline{\mathbb{Q}}_{\ell} )(1) \right). 
\end{align*}
Since 
$\tau_{\zeta', \chi,c} \otimes \rho_{\zeta', \chi,c} \simeq 
 \Ind_{W_K \times L^{\times} U_D ^1} ^{W_K \times D^{\times}} 
 (\tau_{\zeta',\chi,c} \otimes \theta_{\zeta',\chi,c})$ 
and 
$(\mathcal{O}_D ^{\times} \rtimes W_K)^0 \subset 
 W_K \times L^{\times} U_D ^1$, 
we have a non-trivial homomorphism 
\[
 (\tau_{\zeta', \chi,c} \otimes \rho_{\zeta', \chi,c} ) 
 |_{(\mathcal{O}_D ^{\times} \rtimes W_K)^0} 
 \longrightarrow 
 (\tau_{\zeta',\chi,c} \otimes \theta_{\zeta',\chi,c}) 
 |_{(\mathcal{O}_D ^{\times} \rtimes W_K)^0}  
\]
by the Frobenius reciprocity. 
Hence, it suffices to show 
there is a non-trivial homomorphism 
\[
 (\tau_{\zeta',\chi,c} \otimes \theta_{\zeta',\chi,c}) 
 |_{(\mathcal{O}_D ^{\times} \rtimes W_K)^0} 
 \longrightarrow 
 H^1 (\overline{\mathbf{X}}^{\mathrm{c}}_{1,\zeta'} , 
 \overline{\mathbb{Q}}_{\ell} )(1) 
\]
as representations of 
$(\mathcal{O}_D ^{\times} \rtimes W_K)^0$. 
We put 
\[
 W'_K =\{ (\lambda_{\sigma}^{-1} ,\sigma ) \in 
 (\mathcal{O}_D ^{\times} \rtimes W_K)^0 \mid 
 \sigma \in W_K \}. 
\]
We consider 
$U_D ^1$ as a subgroup of 
$(\mathcal{O}_D ^{\times} \rtimes W_K)^0$ 
by identifying 
$d \in U_D ^1$ with 
$(d,1) \in (\mathcal{O}_D ^{\times} \rtimes W_K)^0$. 
Then we have an isomorphism 
\[
 (\tau_{\zeta',\chi,c} \otimes \theta_{\zeta',\chi,c}) 
 |_{W'_K} 
 \longrightarrow 
 H^1 (\overline{\mathbf{X}}^{\mathrm{c}}_{1,\zeta'} , 
 \overline{\mathbb{Q}}_{\ell} )(1)|_{W'_K} 
\]
as representations of $W'_K$ 
by Proposition \ref{descact} 
and the definition of 
$\tau_{\zeta',\chi,c}$ and 
$\theta_{\zeta',\chi,c}$. 
This isomorphism is compatible with 
the action of $U_D ^1$ 
by Proposition \ref{descact}. 
Then this is an isomorphism 
as representations of 
$(\mathcal{O}_D ^{\times} \rtimes W_K)^0$, 
because 
$(\mathcal{O}_D ^{\times} \rtimes W_K)^0$ 
is generated by 
$W'_K$ and $U_D ^1$. 
Thus we have proved that 
\begin{equation}\label{homisom}
 \Hom_{D^{\times}} \bigl( \rho_{\zeta', \chi,c} , 
 \Hom_{GL_2 (K)} (H^1 _{\mathrm{LT}} ,\pi_{\zeta', \chi,c}) \bigr) 
 \simeq \tau_{\zeta', \chi, c}. 
\end{equation}

By \eqref{homisom}, we see that 
$\Hom_{GL_2 (K)} (H^1 _{\mathrm{LT}} ,\pi_{\zeta', \chi,c})$ 
is an irreducible representation of $W_K \times D^{\times}$. 
The group $Q \rtimes (\mathbb{Z}/2\mathbb{Z})$ 
is regarded as a quotient of 
$W_K$ via $\Theta_{\zeta'}$. 
Let $\xi_c \colon D^{\times} \to \overline{\mathbb{Q}}_{\ell}^{\times}$ 
be the character defined by 
$\xi_c (d)=c^{v(\Nrd_{D/K}(d))}$. 
By \eqref{homisom}, 
we have 
\begin{align*}
 \Hom_{D^{\times}} \bigl( \rho_{\zeta', \chi,c} \otimes 
 \xi^{-1}_c , 
 \Hom_{GL_2 (K)} (H^1 _{\mathrm{LT}} ,\pi_{\zeta', \chi,c}) 
 \otimes \phi^{-1}_{c(-2)^{-f/2}} 
 &\otimes (\chi \circ \lambda)^{-1} \otimes \xi^{-1}_c \bigr) \\ 
 &\simeq \tau_{\zeta', \chi, c} 
 \otimes \phi^{-1}_{c(-2)^{-f/2}} 
 \otimes (\chi \circ \lambda)^{-1}. 
\end{align*}
Then we see that 
$\tau_{\zeta', \chi, c} \otimes \phi^{-1}_{c(-2)^{-f/2}}
 \otimes (\chi \circ \lambda)^{-1}$ and 
$\rho_{\zeta', \chi,c} \otimes \xi^{-1}_c$ 
factor through representations of 
$Q \rtimes (\mathbb{Z}/2\mathbb{Z})$ and 
$D^{\times}/(\varpi^{\mathbb{Z}} (1+\varpi\mathcal{O}_D))$ 
respectively, 
where we use 
Lemma \ref{quadind} 
for the first factorization. 
Hence, 
the $(W_K \times D^{\times})$-representation 
\[
 \Hom_{GL_2 (K)} (H^1 _{\mathrm{LT}} ,\pi_{\zeta', \chi,c}) 
 \otimes \phi^{-1}_{c(-2)^{-f/2}} 
 \otimes (\chi \circ \lambda)^{-1} \otimes \xi^{-1}_c 
\]
factors through 
a representation of the finite group 
$(Q \rtimes (\mathbb{Z}/2\mathbb{Z})) \times 
 \bigl( D^{\times}/(\varpi^{\mathbb{Z}} (1+\varpi\mathcal{O}_D))
 \bigr)$. 
Then we have 
\begin{align*}
 \Hom_{GL_2 (K)} (H^1 _{\mathrm{LT}} ,\pi_{\zeta', \chi,c}) 
 \otimes \phi^{-1}_{c(-2)^{-f/2}} 
 &\otimes (\chi \circ \lambda)^{-1} \otimes \xi^{-1}_c \\ 
 &\simeq \bigl( \tau_{\zeta', \chi, c} 
 \otimes \phi^{-1}_{c(-2)^{-f/2}} 
 \otimes (\chi \circ \lambda)^{-1} \bigr) 
 \otimes (\rho_{\zeta', \chi,c} \otimes 
 \xi^{-1}_c ), 
\end{align*}
because 
an irreducible representation of a product of two finite groups is 
isomorphic to 
a tensor product of irreducible representations of 
the two groups. 
Therefore, we have the claim. 
\end{proof}

\section{Local Langlands correspondence}\label{LLCe}
In this section, 
we prove that the correspondence in 
Theorem \ref{realLLC} actually gives the 
local Langlands correspondence. 
After we introduce some notations, 
we give an explicit description of the Artin map 
in Subsection \ref{ssec:ExArt}. 
This enables us to calculate an epsilon factor 
explicitly. 
In Subsection \ref{ssec:ExLLC}, 
we give an explicit description of 
the local Langlands correspondence for 
$\tau_{\zeta', \chi, c}$ using a result 
in Subsection \ref{ssec:ExArt}. 

We write 
$\mathcal{G}_2(K,\overline{\mathbb{Q}}_{\ell})$ 
for the set of equivalent classes of 
two-dimensional Frobenius semisimple Weil--Deligne representations 
of $W_K$ over 
$\overline{\mathbb{Q}}_{\ell}$, and 
$\mathrm{Irr} (\textit{GL}_2 (K), \overline{\mathbb{Q}}_{\ell})$ 
for the set of equivalent classes of 
irreducible smooth representations of 
$\textit{GL}_2 (K)$. 
For 
$\pi \in \mathrm{Irr} (\textit{GL}_2 (K), \overline{\mathbb{Q}}_{\ell})$, 
let $\omega_{\pi}$ denote the central character of 
$\pi$. 
Let 
\[
 \mathrm{LL}_{\ell} \colon 
 \mathcal{G}_2(K,\overline{\mathbb{Q}}_{\ell}) \to 
 \mathrm{Irr} (\textit{GL}_2 (K), \overline{\mathbb{Q}}_{\ell}) 
\]
be the $\ell$-adic Langlands correspondence. 
We follow the normalization in \cite[35.1]{BHLLCGL2}. 
If we take an isomorphism 
$\iota \colon \overline{\mathbb{Q}}_{\ell} \simeq \mathbb{C}$, 
then ${}^{\iota} \tau$ and 
${}^{\iota} \pi$ denote 
the representations over $\mathbb{C}$ associated to 
$\tau$ and $\pi$ by $\iota$ respectively 
for 
$\tau \in \mathcal{G}_2(K,\overline{\mathbb{Q}}_{\ell})$ and 
$\pi \in \mathrm{Irr} (\textit{GL}_2 (K), \overline{\mathbb{Q}}_{\ell})$. 
We use similar notations also over 
a finite extension of $K$. 

\begin{rem}\label{ellLLC}
The $\ell$-adic Langlands correspondence 
$\mathrm{LL}_{\ell}$ satisfies that 
\[
 \omega_{\mathrm{LL}_{\ell}(\tau)} \circ \mathrm{Art}_K ^{-1} = 
 (\det \tau ) \otimes \lvert \cdot \rvert^{-1} 
\]
for 
$\tau \in \mathcal{G}_2(K,\overline{\mathbb{Q}}_{\ell})$. 
If we take an isomorphism 
$\iota \colon \overline{\mathbb{Q}}_{\ell} \simeq 
 \mathbb{C}$, 
then we have 
\[
 \varepsilon \left( {}^{\iota} \tau ,s, \psi \right) = 
 \varepsilon \left( {}^{\iota} \mathrm{LL}_{\ell}(\tau) , 
 s+\frac{1}{2}, \psi \right) 
\]
for any non-trivial additive character 
$\psi \colon K \to \mathbb{C}^{\times}$. 
\end{rem}

For a finite extension $K'$ of $K$, 
we define an additive character 
$\psi_{K'} \colon K' \to \overline{\mathbb{Q}}_{\ell}^{\times}$ 
by $\psi_{K'} =\psi_K \circ \Tr_{K'/K}$, and 
let $v_{K'}$ be the normalized discrete valuation of $K'$ 
that sends a uniformizer to $1$. 

\subsubsection*{Definition of $F$, $L'$, $\epsilon_{F/K}$, $\Lambda_{F,\zeta'}$ and $\pi_{F,\zeta'}$}

We take $\zeta' \in k^{\times}$. 
We simply write $\pi_{\zeta'}$, $\Lambda_{\zeta'}$ 
and $\tau_{\zeta'}$ for 
$\pi_{\zeta',1,1}$, $\Lambda_{\zeta',1,1}$ 
and $\tau_{\zeta',1,1}$ respectively. 
We put 
\[
 F=K(\zeta'' \varpi^{1/3}) \quad 
 \textrm{and} \quad 
 L'=F(\varphi). 
\]
We define 
$\mathfrak{J}_F ,\mathfrak{P}_F \subset M_2 (\mathcal{O}_F)$ 
similarly to $\mathfrak{J}$ and $\mathfrak{P}$ as in 
\eqref{JPdef}. 
We put $U_{\mathfrak{J}_F}^i =1+\mathfrak{P}_F ^i$ 
for any positive integer $i$. 
We consider $L'$ as an $F$-subalgebra of 
$M_2 (F)$ similarly as \eqref{embL}. 
We put 
\[
 \epsilon_{F/K} =(-1)^f. 
\]
Let $\pi_{F,\zeta'}$ be the 
tame lifting of $\pi_{\zeta'}$ 
to $F$. 
See \cite[46.5 Definition]{BHLLCGL2} 
for the tame lifting. 
We define 
a character 
$\Lambda_{F,\zeta'} \colon 
 {L'}^{\times} U_{\mathfrak{J}_F} ^2 \to 
 \overline{\mathbb{Q}}_{\ell}^{\times}$ by 
\begin{align*}
 \Lambda_{F,\zeta'} (x) &=\epsilon_{F/K} ^{v_{L'} (x)} 
 \Lambda_{\zeta'} (\Nr_{L'/L} (x)) \quad 
 \textrm{for $x \in L'^{\times}$,} \\ 
 \Lambda_{F,\zeta'} (x) &= 
 (\psi_F \circ \tr ) 
 \left( \hat{\zeta}'^{-2} \varphi^{-1} (x-1) \right) \quad 
 \textrm{for $x \in U_{\mathfrak{J}_F} ^2$.} 
\end{align*}
Then we have 
$\pi_{F,\zeta'} =
 \mathrm{c\mathchar`-Ind}_{{L'}^{\times} U_{\mathfrak{J}_F}^2} 
 ^{\textit{GL}_2(F)} \Lambda_{F,\zeta'}$ 
by \cite[46.3 Proposition]{BHLLCGL2} and 
the construction of the tame lifting. 

We will describe the 
restriction of $\tau_{\zeta'}$ 
to $W_F$. 
The field $F$ corresponds to 
the subgroup 
$Q_8 \rtimes \mathbb{Z}$ of 
$Q \rtimes \mathbb{Z}$. 

\subsubsection*{Definition of $\delta_2$, $\delta_4$ and $\theta_2$} 

First, we consider the case where 
$f$ is even. 
We put $h_0 (x)=x^2-x$. 
Then we have 
\[
 h_0 (\delta^2 -\delta) \equiv 1/(\zeta'' \varpi^{1/3}) \pmod{3/4}. 
\]
Hence we can take 
$\delta_2 \in F(\delta)$ 
such that 
$h_0 (\delta_2)=1/(\zeta'' \varpi^{1/3})$ 
and 
$\delta_2 \equiv \delta^2 -\delta \pmod{3/4}$
by Newton's method. 
Similarly, 
we can take $\delta_4 \in F(\delta)$ 
such that 
$\delta_4 ^2 - \delta_4=\delta_2$ 
and 
$\delta_4 \equiv \delta \pmod{3/4}$. 
Then we have 
$F(\delta_4) =F(\delta)$. 
Further, we can take 
$\theta_2 \in F(\theta)$ such that 
$\theta_2 ^2 -\theta_2 =\delta_4 ^3$ 
and 
$\theta_2 \equiv \theta \pmod{7/12}$. 
We have $F(\theta_2) =F(\theta)$. 
Then 
$F(\delta_2)$ corresponds to 
the subgroup 
$C_4 \times \mathbb{Z}$ of 
$Q \times \mathbb{Z}$. 

Next, we consider the case where 
$f$ is odd. 
We put $h_1 (x)=x^2-x +1$. 
Then we have 
\[
 h_1 (\delta^2 -\delta +\zeta_3) \equiv 
 1/(\zeta'' \varpi^{1/3}) \pmod{3/4}. 
\]
Hence we can take 
$\delta_2 \in F(\zeta_3, \delta)$ 
such that 
$h_1 (\delta_2)=1/(\zeta'' \varpi^{1/3})$ 
and 
$\delta_2 \equiv \delta^2 -\delta +\zeta_3 \pmod{3/4}$
by Newton's method. 
Similarly, 
we can take $\delta_4 \in F(\zeta_3 ,\delta)$ 
such that 
$\delta_4 ^2 - \delta_4 +\zeta_3 =\delta_2$ 
and 
$\delta_4 \equiv \delta \pmod{3/4}$. 
Then we have 
$F(\zeta_3 ,\delta_4) =F(\zeta_3 ,\delta)$. 
Further, we can take 
$\theta_2 \in F(\zeta_3, \theta)$ such that 
$\theta_2 ^2 -\theta_2=\delta_4 ^3$ 
and 
$\theta_2 \equiv \theta \pmod{7/12}$. 
We have $F(\zeta_3,\theta_2) =F(\zeta_3,\theta)$. 
Then 
$F(\delta_2)$ corresponds to 
the subgroup 
$C$ of 
$Q \rtimes \mathbb{Z}$. 

We note that $v(\delta_2)=-1/6$ for any $f$. 

\subsubsection*{Definition of $E$, $\phi_{\zeta'}$ and $\varkappa_{E/F}$}

We put $E=F(\delta_2)$. 
The image of $W_E$ under $\Theta_{\zeta'}$ equals $C$. 
Let $\phi_{\zeta'}$ be the character of 
$W_E$ induced from 
$\phi'$ by $\Theta_{\zeta'}$. 
Then we have 
\[
 \tau_{\zeta'} |_{W_F} \simeq 
 \Ind_{W_E}^{W_F} \phi_{\zeta'}. 
\]
We consider $\phi_{\zeta'}$ as a character of 
$E^{\times}$ by the Artin reciprocity map 
$\mathrm{Art}_E$. 
For a finite extension $K'$ of $K$ 
and integer $i$, 
we write $\mathfrak{p}_{K'}$ for 
the maximal ideal of 
$\mathcal{O}_{K'}$, and 
put $U_{K'}^i =1+\mathfrak{p}_{K'}^i$. 
Let $E_m$ be the unramified extension over $E$ 
of degree $m$ for a positive integer $m$. 
Let 
$\varkappa_{E/F}$ be the character of 
$F^{\times}$ with kernel 
$\Nr_{E/F} (E^{\times})$. 

\subsection{Explicit Artin reciprocity law}\label{ssec:ExArt}

The results in this subsection will be used 
in the proof of Proposition \ref{indLLC}. 

For a Galois group $G$ of a finite Galois extension of 
a non-Archimedean field, let 
$G_s$ and $G^t$ be the ramification subgroups of $G$ 
with lower numbering and upper numbering respectively. 
Note that 
\begin{equation}\label{eq:KerTr}
 \Ker \Tr_{k/\mathbb{F}_2} = 
 \{ \xi +\xi^2 \mid \xi \in k \}. 
\end{equation}

\begin{lem}\label{phivarkappa}
We have 
\[
 \phi_{\zeta'} (1+x) =\psi_E (\delta_2 ^3 x) 
\]
for $x \in \mathfrak{p}_E ^2$ and 
\[
 \varkappa_{E/F} (1+y) =
 \psi_F \left( (\zeta''\varpi^{1/3})^{-1} y \right) 
\]
for $y \in \mathfrak{p}_F$. 
\end{lem}
\begin{proof}
We prove the first statement only in the case where 
$f$ is odd. 
It is easier to prove the first statement 
in the case where $f$ is even. 

We put 
$G=\Gal (E_2 (\theta)/E)$. 
For $\sigma \in I_E$, we can show that 
\[
 v \biggl( \sigma\biggl(\frac{\delta}{\theta}\biggl) 
 -\frac{\delta}{\theta} \biggr) =
 \begin{cases}
 \frac{1}{12} &\quad 
 \textrm{if $\zeta_{3,\sigma}=1$, $\nu_{\sigma} =1$,}\\ 
 \frac{1}{6} &\quad 
 \textrm{if $\zeta_{3,\sigma}=1$, $\nu_{\sigma}=0$, $\mu_{\sigma} =1$} 
 \end{cases}
\] 
by the definition of 
$\zeta_{\sigma}$, $\nu_{\sigma}$ and $\mu_{\sigma}$. 
Then we have 
\[
 \Gal (E_2 (\theta)/E_2) =G_0 =G_1 \supset 
 \Gal (E_2 (\theta)/E_2 (\delta)) =G_2 =G_3 \supset 
 \{1 \}=G_4
\] 
and 
\[
 G^t = 
 \begin{cases}
 \Gal (E_2 (\theta)/E_2 ) &\quad 
 \textrm{if $0 \leq t \leq 1$,}\\ 
 \Gal (E_2 (\theta)/E_2 (\delta)) &\quad 
 \textrm{if $1 < t \leq 2$,}\\ 
 \{ 1 \} &\quad 
 \textrm{if $2 < t$.} 
 \end{cases}
\] 
Then the restriction of $\phi_{\zeta'}$ to 
$U_E ^2$ equals the composite 
\[
 U_E ^2 \twoheadrightarrow 
 U_E ^2 /(U_E ^3 \Nr_{E_2(\theta)/E}(U_{E(\theta)}^3)) 
 \xrightarrow{\sim} 
 \Gal (E_2 (\theta)/E_2 (\delta)) 
 \simeq Z \xrightarrow{\phi |_Z} 
 \overline{\mathbb{Q}}_{\ell}^{\times} 
\]
by 
\cite[XV \S 2 Corollaire 3 au Th\'{e}or\`{e}me 1]{SerCL}. 
We define 
$N_2 \colon k_2 \to k$ by 
$N_2 (x) =\Tr_{k_2/k}(x)^2 +\Tr_{k_2/k}(x)$. 
Then we can check that 
\[
 \Nr_{E_2(\theta)/E} \colon 
 U_{E_2(\theta)}^3 /U_{E_2(\theta)}^4 \to 
 U_E ^2 /U_E ^3 
\] 
becomes 
$N_2 \colon k_2 \to k$ 
under the identifications 
\[ 
 U_{E_2(\theta)}^3 /U_{E_2(\theta)}^4 \simeq k_2 ; \ 
 1+\theta^{-1} x \mapsto \bar{x} \quad 
 \textrm{and} \quad 
 U_E^2 /U_E^3 \simeq k ; \ 
 1+\delta_2^{-2} x \mapsto \bar{x}. 
\] 
Therefore we have 
\begin{equation}\label{eq:phichi}
 \phi_{\zeta'} (1+x) =(\chi_2 \circ \Tr_{k/\mathbb{F}_2} ) 
 (\overline{\delta_2 ^2x} ) 
\end{equation}
for $x \in \mathfrak{p}_E^2$, 
because 
$\Image N_2 =\Ker \Tr_{k/\mathbb{F}_2}$ by \eqref{eq:KerTr}. 
Since we have 
$(\chi_2 \circ \Tr_{k/\mathbb{F}_2} ) 
 (\bar{x} ) =\psi_E (\delta_2 x)$ 
for $x \in \mathcal{O}_E$, 
the first statement follows. 

We can prove the second statement similarly. 
\end{proof}

\begin{lem}\label{detNr}
We consider $\hat{\zeta}'^{-2} \varphi^{-1}$ as an element of 
$\textit{GL}_2 (F)$ by the embedding \eqref{embL}. 
Then we have 
\begin{alignat*}{2}
 \det (\hat{\zeta}'^{-2} \varphi^{-1} ) &\equiv \Nr_{E/F} (\delta_2 ^3) 
 & \quad &\mod U_F ^1 , \\ 
 \tr (\hat{\zeta}'^{-2} \varphi^{-1} ) &\equiv 
 (\zeta'' \varpi^{\frac{1}{3}})^{-1} + \Tr_{E/F} (\delta_2 ^3 ) 
 & \quad &\mod \mathcal{O}_F . 
\end{alignat*}
\end{lem}
\begin{proof}
We have 
\[
 \det (\hat{\zeta}'^{-2} \varphi^{-1} )=
 -\hat{\zeta}'^{-4} \varpi^{-1} = - \zeta''^{-3} \varpi^{-1} 
\]
by \eqref{eq:zeta''}. On the other hand, we have 
\[
 \Nr_{E/F} (\delta_2 ^3) = \Nr_{E/F} (\delta_2)^3 
 =
 \begin{cases}
 - (\zeta''^3 \varpi)^{-1} & \textrm{if $f$ is even,}\\ 
 (1-(\zeta'' \varpi^{1/3})^{-1})^3 & \textrm{if $f$ is odd.} 
 \end{cases}
\]
Hence, we have the first congruence. 
We have $\tr (\hat{\zeta}'^{-2} \varphi^{-1} )=0$. 
On the other hand, we have 
\begin{align*}
 \Tr_{E/F} (\delta_2 ^3) &= \Tr_{E/F} (\delta_2)^3 -
 3 \Nr_{E/F} (\delta_2) \Tr_{E/F} (\delta_2) \\ 
 &=1- 3 \Nr_{E/F} (\delta_2) 
 = 
 \begin{cases}
 1 + 3 (\zeta'' \varpi^{1/3})^{-1} & \textrm{if $f$ is even,}\\ 
 -2 + 3 (\zeta'' \varpi^{1/3})^{-1} & \textrm{if $f$ is odd.} 
 \end{cases}
\end{align*}
Hence, we have the second congruence. 
\end{proof}

Let $\mathcal{E}'$ be the elliptic curve over $\mathbb{F}_2$ 
defined by $z^2 +z =w^3 +w$. 
Let $\alpha_1, \alpha_2 \in \overline{\mathbb{Q}}_{\ell}$ 
be the roots of 
$x^2 +2x +2 =0$. 

\begin{lem}\label{ratell}
We have 
\begin{align*}
 \lvert \mathcal{E} (\mathbb{F}_q) \rvert 
 &=q+1 -\left( (-2)^{1/2} \right)^f -\left( -(-2)^{1/2} \right)^f , \\  
 \lvert \mathcal{E}' (\mathbb{F}_q) \rvert 
 &=q+1 -\alpha_1^f -\alpha_2^f. 
\end{align*}
\end{lem}
\begin{proof}
We have 
$\tr (\mathrm{fr}_2 ^* ;H^1(\mathcal{E}_{k^{\mathrm{ac}}},\overline{\mathbb{Q}}_{\ell}) )=0$ and 
$\tr (\mathrm{fr}_4 ^* ;H^1(\mathcal{E}_{k^{\mathrm{ac}}},\overline{\mathbb{Q}}_{\ell}) )=-4$ 
as in the proof of Lemma \ref{Trtau2}. 
Hence we obtain 
\[
 \tr (\mathrm{fr}_q ^* ;H^1(\mathcal{E}_{k^{\mathrm{ac}}},\overline{\mathbb{Q}}_{\ell}) )= 
 \left( (-2)^{1/2} \right)^f +\left( -(-2)^{1/2} \right)^f. 
\]
The first claim follows from this 
and the Lefschetz trace formula. 
The second claim is proved similarly 
by 
$\tr(\mathrm{fr}_2^\ast;H^1(\mathcal{E}'_{k^{\mathrm{ac}}},\overline{\mathbb{Q}}_{\ell}))=-2$ 
and 
$\tr(\mathrm{fr}_4^\ast;H^1(\mathcal{E}'_{k^{\mathrm{ac}}},\overline{\mathbb{Q}}_{\ell}))=0$. 
\end{proof}

If $f$ is odd, 
then the map $\Theta_{\zeta'}$ induces an 
isomorphism 
$W(E^{\mathrm{ur}}(\theta)/E) \simeq C$, and 
we write $\mathfrak{a}_E$ for the composite 
\[
 E^{\times} \xrightarrow{\mathrm{Art}_E} W_E^{\mathrm{ab}} 
 \twoheadrightarrow W(E^{\mathrm{ur}}(\theta)/E) \simeq C. 
\]
\begin{lem}\label{reclaw}
We assume that $f$ is odd. 
Let $n_f$ and $m_f$ be the integers such that 
$1 \leq n_f, m_f \leq 2$, 
$n_f \equiv (f+1)/2 \bmod 2$ and 
$m_f \equiv (f^2+7)/8 \bmod 2$. 
Then we have 
$\mathfrak{a}_E (\delta_2)=
 (g(1,\bar{\zeta}_3 ^{2n_f} ,\bar{\zeta}_3 ^{m_f}),-1)$. 
\end{lem}
\begin{proof}
Let $\overline{C}$ be the image of $C$ in 
$Q_8 \rtimes (\mathbb{Z}/2\mathbb{Z})$. 
Then $\overline{C}$ is a cyclic group of order $8$. 
Let $\overline{\mathfrak{a}}_E$ be the composite of 
$\mathfrak{a}_E$ with the natural projection 
$C \to \overline{C}$. 
It suffices to show that 
\[
 \overline{\mathfrak{a}}_E (\delta_2 )
 =(g(1,\bar{\zeta}_3 ^{2n_f} ,\bar{\zeta}_3 ^{m_f}),1), 
\]
because we know that the second component of 
$\mathfrak{a}_E (\delta_2 )$ is $-1$. 
We note that the isomorphism 
$W(E^{\mathrm{ur}}(\theta)/E) \simeq C$ induces 
$\Gal (E_2 (\theta_2) /E) \simeq \overline{C}$. 
By this isomorphism, 
we consider 
$(g(1,\bar{\zeta}_3 ^{2n_f} ,\bar{\zeta}_3 ^{m_f}),1)$ as 
an element of 
$\Gal (E_2 (\theta_2) /E)$. 

We write $E_{(0)}$ for $E$ 
in the mixed characteristic case, 
and 
$E_{(p)}$ for $E$ 
in the equal characteristic case. 
We use similar notations for other fields 
and elements of the fields. 
Then we have the isomorphism 
\[
 E_{(0)}^{\times} /U^3_{E_{(0)}} \simeq 
 E_{(p)}^{\times} /U^3_{E_{(p)}} 
 ;\ 
 \hat{\xi}_0 +\hat{\xi}_1 \delta_{2,(0)}^{-1} 
 +\hat{\xi}_2 \delta_{2,(0)}^{-2} \mapsto 
 \xi_0 +\xi_1 \delta_{2,(p)}^{-1} 
 +\xi_2 \delta_{2,(p)}^{-2} 
\] 
where 
$\xi_0, \xi_1, \xi_2 \in k \subset E_{(p)}$. 
This isomorphism induces an isomorphism 
\[
 \bigl( \Gal (E_{(0)}^{\mathrm{sep}}/E_{(0)}) / 
 \Gal (E_{(0)}^{\mathrm{sep}}/E_{(0)})^3 \bigr)^{\mathrm{ab}} 
 \simeq 
 \bigl( \Gal (E_{(p)}^{\mathrm{sep}}/E_{(p)}) / 
 \Gal (E_{(p)}^{\mathrm{sep}}/E_{(p)})^3 \bigr)^{\mathrm{ab}} 
\]
by \cite[(3.5.2)]{Delp0}. 
It further induces an isomorphism 
\[
 \Gal \bigl( E_{2,(0)}(\theta_{2,(0)})/E_{(0)} \bigr) \simeq 
 \Gal \bigl( E_{2,(p)}(\theta_{2,(p)})/E_{(p)} \bigr). 
\]
Then we have a commutative diagram 
\begin{align*}
 &\xymatrix{
 E_{(0)}^{\times} /U^3_{E_{(0)}}
 \ar@{->}[r]^{\mathrm{Art}_{E_{(0)}} \hspace*{3em}} 
 \ar@{-}[d]^{\! \rotatebox{90}{$\sim$}}  & 
 \Gal \bigl( E_{2,(0)}(\theta_{2,(0)})/E_{(0)} \bigr) 
 \ar@{->}[r]^{\hspace*{4.5em} \sim} 
 \ar@{-}[d]^{\! \rotatebox{90}{$\sim$}}  & 
 \overline{C} 
 \ar@{=}[d] 
 \\
 E_{(p)}^{\times} /U^3_{E_{(p)}}
 \ar@{->}[r]^{\mathrm{Art}_{E_{(p)}} \hspace*{3em}} & 
 \Gal \bigl( E_{2,(p)}(\theta_{2,(p)})/E_{(p)} \bigr) 
 \ar@{->}[r]^{\hspace*{4.5em} \sim} & 
 \overline{C} 
 }
\end{align*}
by \cite[(3.6.1)]{Delp0} and 
the construction of the isomorphisms. 
Therefore, 
it suffices to show that 
$\overline{\mathfrak{a}}_E (\delta_2 )
 =(g(1,\bar{\zeta}_3 ^{2n_f} ,\bar{\zeta}_3 ^{m_f}),1)$ 
in the equal characteristic case. 

We assume that 
the characteristic of $E$ is $p$. 
We define the central division algebra $D_g$ 
over $E$ of degree $64$ by 
\[
 D_g = \bigoplus_{i=0} ^7 E_2 (\theta_2) s^i 
\]
where 
$s^8 =\delta_2$ and 
$s a s^{-1} =
 (g(1,\bar{\zeta}_3 ^{2n_f} ,\bar{\zeta}_3 ^{m_f}),1) (a)$ 
for $a \in E_2 (\theta_2)$. 
Let $\sigma_q \in \Gal(E_8 /E)$ be the lift of 
$\Fr_q$. 
We define 
the central division algebra $D_{\sigma}$ 
over $E$ of degree $64$ by 
\[
 D_{\sigma} = \bigoplus_{i=0} ^7 E_8 t^i 
\]
where 
$t^8 =\delta_2$ and 
$t a t^{-1} =\sigma_q (a)$ 
for $a \in E_8$. 
By the construction of 
the Artin reciprocity map, 
it suffices to show $D_g \simeq D_{\sigma}$ 
to prove the claim. 
To show this isomorphism, 
it suffices to find 
$s' ,\delta'_4 ,\theta'_2 \in D_{\sigma}$ such that 
\begin{align*}
 &{s'}^8=\delta_2,\quad 
 {\delta'_4}^2 -\delta'_4 +\zeta_3 =\delta_2,\quad 
 {\theta'_2}^2 -\theta'_2 ={\delta'_4}^3,\quad 
 \delta'_4 \theta'_2 =\theta'_2 \delta'_4 ,\\ 
 &s' \zeta_3 {s'}^{-1} =\zeta_3 ^2,\quad 
 s' \delta'_4 {s'}^{-1} = \delta'_4 +\zeta_3 ^{n_f},\quad 
 s' \theta'_2 {s'}^{-1} = \theta'_2 + 
 \zeta_3 ^{2n_f} \delta'_4 +\zeta_3 ^{m_f}. 
\end{align*}

We put $s' =t$. 
Then we have 
${s'}^8=\delta_2$ and 
$s' \zeta_3 {s'}^{-1} =\zeta_3 ^2$. 
We take $a_0 \in \mu_{q^4 -1} (E_4)$ such that 
$a_0 ^2 -a_0 =\zeta_3$. 
We put 
\[
 \delta'_4 =a_0 +t^2 +t^4. 
\]
Then we can check that 
${\delta'_4}^2 -\delta'_4 +\zeta_3 = \delta_2$ 
using 
$t^2 a_0 t^{-2} = a_0 + 1$. 
We can check also that 
$t \delta'_4 t^{-1} = \delta'_4 + \zeta_3 ^{n_f}$ 
using 
$t a_0 t^{-1} = a_0 + \zeta_3 ^{n_f}$. 

We take $b_0 \in \mu_{q^8 -1} (E_8)$ 
and 
$b_4 \in \mu_{q^4 -1} (E_4)$ 
such that 
$b_0 ^2 -b_0 =a_0 \zeta_3^2 +\zeta_3$ 
and 
$b_4 ^2=a_0$. 
We put 
\[
 \theta'_2 =b_0 +(a_0 +\zeta_3)t^2 +b_4 t^4 +t^6. 
\]
Then we can check that 
${\theta'_2}^2 -\theta'_2 ={\delta'_4}^3$, 
$\delta'_4 \theta'_2 =\theta'_2 \delta'_4$ and 
$t \theta'_2 t^{-1} = \theta'_2 + 
 \zeta_3 ^{2n_f} \delta'_4 +\zeta_3 ^{m_f}$ 
using 
$t b_0 t^{-1} =b_0 +a_0 \zeta_3 ^{2n_f} + \zeta_3 ^{m_f}$ and 
$t b_4 t^{-1} =b_4 +\zeta_3^{2n_f}$. 
Therefore, we have proved the claim. 
\end{proof}

\subsection{Explicit local Langlands correspondence}\label{ssec:ExLLC}

In the next proposition, 
we show that 
$\tau_{\zeta'} |_{W_F}$ corresponds to 
$\pi_{F,\zeta'}$ 
under the local Langlands correspondence 
by calculating their epsilon factors. 
We will show a correspondence over $K$ in 
Theorem \ref{expLLC} using 
the correspondence over $F$ and a construction of 
the local Langlands correspondence for 
primitive representations in \cite[50.3]{BHLLCGL2}. 

\begin{prop}\label{indLLC}
We have 
$\mathrm{LL}_{\ell} (\tau_{\zeta'} |_{W_F}) =\pi_{F,\zeta'}$. 
\end{prop}
\begin{proof}
We put 
$\mathrm{LL}_{\ell} (\tau_{\zeta'}) =\pi'_{\zeta'}$ 
and 
$\mathrm{LL}_{\ell} (\tau_{\zeta'} |_{W_F}) =\pi'_{F,\zeta'}$. 
We want to show that 
$\pi'_{F,\zeta'} = \pi_{F,\zeta'}$. 

By Lemma \ref{phivarkappa}, Lemma \ref{detNr} and 
\cite[44.7 Proposition]{BHLLCGL2}, 
the representation 
$\pi'_{F,\zeta'}$ contains the ramified 
simple stratum 
$(\mathfrak{J}_F, 3, \hat{\zeta}'^{-2} \varphi^{-1} )$. 
Then the representation 
$\pi'_{\zeta'}$ contains the ramified 
simple stratum 
$(\mathfrak{J}, 1, \hat{\zeta}'^{-2} \varphi^{-1} )$ 
by the construction of 
$\pi'_{\zeta'}$ in \cite[50.3]{BHLLCGL2}. 
Therefore we have 
\[
 \pi'_{\zeta'} =
 \mathrm{c\mathchar`-Ind}_{L^{\times} U_{\mathfrak{J}}^1} 
 ^{\textit{GL}_2(K)} \Lambda'_{\zeta'} 
\]
for a character 
$\Lambda'_{\zeta'} \colon 
 L^{\times} U_{\mathfrak{J}} ^1 \to 
 \overline{\mathbb{Q}}_{\ell}^{\times}$ such that 
$\Lambda'_{\zeta'} =\Lambda_{\zeta'}$ on 
$U_{\mathfrak{J}} ^1$. 

Let $1_F$ denote the trivial character of $W_F$. 
We put 
\[
 \varkappa_{F/K} =\det \Ind_{W_F}^{W_K} 1_F. 
\]
Then $\varkappa_{F/K} |_{W_F} =1_F$ if $f$ is even, and 
$\varkappa_{F/K} |_{W_F}$ is the unramified character of 
order two if $f$ is odd. 
Hence, the definition of 
$\epsilon_{F/K}$ in \cite[46.3]{BHLLCGL2} 
coincides with that in this paper. 
By \cite[46.3 Proposition]{BHLLCGL2}, 
we have 
\[
 \pi'_{F,\zeta'} =
 \mathrm{c\mathchar`-Ind}_{{L'}^{\times} U_{\mathfrak{J}_F}^2} 
 ^{\textit{GL}_2(F)} \Lambda'_{F,\zeta'} 
\]
for a character 
$\Lambda'_{F,\zeta'} \colon 
 {L'}^{\times} U_{\mathfrak{J}_F} ^2 \to 
 \overline{\mathbb{Q}}_{\ell}^{\times}$ such that 
$\Lambda'_{F,\zeta'} =\Lambda_{F,\zeta'}$ on 
$U_{\mathfrak{J}_F} ^2$ and 
\[
 \Lambda'_{F,\zeta'} (x) =\epsilon_{F/K} ^{v_{L'} (x)} 
 \Lambda'_{\zeta'} (\Nr_{L'/L} (x))
\] 
for 
$x \in L'^{\times}$. 
Hence we have 
$\Lambda'_{F,\zeta'} (x)=1$ 
for $x \in U_{L'} ^1$, 
because $\Lambda'_{\zeta'} (x)=1$ 
for $x \in U_L ^1$. 
Then we see that 
$\Lambda'_{F,\zeta'} =\Lambda_{F,\zeta'}$ 
on 
$U_{L'} ^1 F^{\times} U_{\mathfrak{J}_F} ^2$, 
because 
$\Lambda'_{F,\zeta'} =\Lambda_{F,\zeta'}$ 
on $F^{\times}$ 
by Remark \ref{ellLLC} and Lemma \ref{dettau}. 

We define 
$\kappa_F \colon F^{\times} \to \overline{\mathbb{Q}}_{\ell}^{\times}$ 
by $\kappa_F (x)=(-1)^{v_F (x)}$. 
Since 
$\Lambda'_{F,\zeta'} =\Lambda_{F,\zeta'}$ 
on 
$U_{L'} ^1 F^{\times} U_{\mathfrak{J}_F} ^2$, 
we know that 
$\Lambda'_{F,\zeta'} =\Lambda_{F,\zeta'}$ or 
$\Lambda'_{F,\zeta'} =\Lambda_{F,\zeta'} \otimes (\kappa_F \circ \det)$. 
We take an isomorphism 
$\iota \colon \overline{\mathbb{Q}}_{\ell} \simeq \mathbb{C}$. 
Then, to show $\Lambda'_{F,\zeta'} =\Lambda_{F,\zeta'}$, 
it suffices to show that 
\[
 \varepsilon ({}^{\iota} \pi'_{F,\zeta'}, 1/2, \iota \circ \psi_F) = 
 \varepsilon ({}^{\iota} \pi_{F,\zeta'}, 1/2, \iota \circ \psi_F). 
\]
We note that 
we have already known this equality up to sign. 

In the sequel of this proof, 
we identify 
$\overline{\mathbb{Q}}_{\ell}$ with $\mathbb{C}$ 
by $\iota$, and omit to write $\iota$. 
By \cite[25.5 Corollary]{BHLLCGL2}, 
we obtain 
$\varepsilon (\pi_{F,\zeta'}, 1/2, \psi_F) 
 =-\epsilon_{F/K}$ 
using that 
$5$ is the least integer $m \geq 0$ such that 
$U_{\mathfrak{J}_F} ^{m+1} \subset \Ker \Lambda_{F,\zeta'}$. 
On the other hand, we have 
\[
 \varepsilon (\pi'_{F,\zeta'}, 1/2, \psi_F) 
 =\varepsilon (\tau_{\zeta'}|_{W_F}, 0, \psi_F) 
 =q^{\frac{3}{2}} \varepsilon (\tau_{\zeta'}|_{W_F}, 1/2, \psi_F), 
\]
because the conductor of $\tau_{\zeta'}|_{W_F}$ 
is three. 
Hence, it suffices to show that 
\[
 \varepsilon (\tau_{\zeta'}|_{W_F},1/2,\psi_F) =-\epsilon_{F/K} q^{-3/2}. 
\]

Let $\lambda_{E/F} (\psi_F)$ be the 
Langlands constant of $E$ over $F$ with respect to $\psi_F$ 
(\cf \cite[30.4]{BHLLCGL2}). 
Let $1_E$ denote the trivial character of $W_E$. 
Then we have 
\[
 \lambda_{E/F} (\psi_F) = 
 \varepsilon (\Ind_{W_E}^{W_F} 1_E, 1/2,\psi_F) 
 \varepsilon (1_E, 1/2,\psi_E)^{-1} = 
 \varepsilon (\varkappa_{E/F}, 1/2,\psi_F) = 
 \varkappa_{E/F}(\zeta''\varpi^{\frac{1}{3}}) 
 =\epsilon_{F/K}, 
\] 
where we use 
$\Ind_{W_E}^{W_F} 1_E \simeq 1_F \oplus \varkappa_{E/F}$ and 
\cite[23.5 Lemma 1 and Proposition]{BHLLCGL2} 
at the second equality, 
Lemma \ref{phivarkappa} and 
\cite[(23.6.2) and 23.6 Proposition]{BHLLCGL2} 
at the third equality, 
and 
Lemma \ref{phivarkappa} and 
the equality 
\[
 \Nr_{E/F} (\delta_2)=
 \begin{cases}
 -1/(\zeta''\varpi^{1/3}) &\quad \textrm{if $f$ is even}, \\ 
 -1/(\zeta''\varpi^{1/3}) +1&\quad \textrm{if $f$ is odd} 
 \end{cases}
\] 
at the last equality. 
Therefore, it suffices to show that 
\[
 \varepsilon (\phi_{\zeta'}, 1/2, \psi_E) =-q^{-3/2}, 
\]
because we have 
$\varepsilon (\tau_{\zeta'}|_{W_F}, 1/2, \psi_F) = 
 \varepsilon (\phi_{\zeta'}, 1/2, \psi_E) 
 \lambda_{E/F} (\psi_F)$. 

We define 
$\psi'_E$ by 
$\psi'_E (x)=\psi_E (\delta_2 x)$ for $x \in E^{\times}$. 
Then $\psi_E '$ has level one (\cf \cite[1.7 Definition]{BHLLCGL2}), 
and we have 
\begin{align*}
 \varepsilon (\phi_{\zeta'}, 1/2, \psi_E) = 
 \phi_{\zeta'} (\delta_2)^{-1} \varepsilon (\phi_{\zeta'}, 1/2, \psi'_E ) &= 
 q^{-\frac{1}{2}} \phi_{\zeta'} (\delta_2)^{-1} 
 \sum_{y \in U^1_E / U^2_E} \phi_{\zeta'} (\delta_2^2 y)^{-1} 
 \psi'_E (\delta_2^2 y) \\ 
 &=
 q^{-\frac{1}{2}} \phi_{\zeta'} (\delta_2^3)^{-1} 
 \sum_{x \in \mathfrak{p}_E / \mathfrak{p}^2_E} 
 \phi_{\zeta'} (1+x)^{-1} 
 \psi_E \left( \delta_2^3 (1+x) \right) \\ 
 &=
 q^{-\frac{1}{2}} \phi_{\zeta'} (\delta_2^3)^{-1} 
 \sum_{x \in \mathfrak{p}_E / \mathfrak{p}^2_E} 
 \phi_{\zeta'} (1+x)^{-1} 
 \psi_E (\delta_2^3 x) 
\end{align*}
by \cite[23.5 Lemma 1, (23.6.2) and (23.6.4)]{BHLLCGL2} and 
$\psi_E (\delta_2^3)=1$. 
Therefore it suffices to show that 
\begin{equation}\label{aim}
 \phi_{\zeta'} (\delta_2^3)^{-1} 
 \sum_{x \in \mathcal{O}_E / \mathfrak{p}_E} 
 \phi_{\zeta'} (1+\delta_2^{-1}x)^{-1} 
 \psi_E (\delta_2^2 x) = -q^{-1}. 
\end{equation}
Note that we have already known this equality up to sign. 

First, we consider the case where $f$ is even. 
Then we have 
\[
 \phi_{\zeta'} (\delta_2^3) =(-2)^{3f/2} \phi_{\zeta'} (-1) 
 =(-2)^{3f/2} 
\]
by 
\[
 \Nr_{E(\theta_2)/E} (\theta_2)= \Nr_{E(\delta_4)/E} (-\delta_4^3) 
 =\Nr_{E(\delta_4)/E} (\delta_4)^3 
 =-\delta_2^3 
\]
and Lemma \ref{phivarkappa} with $x=-2$. 
Hence, it suffices to show 
\[
 \sum_{x \in \mathcal{O}_E / \mathfrak{p}_E} 
 \phi_{\zeta'} (1+\delta_2^{-1}x)^{-1} 
 \psi_E (\delta_2^2 x) = -(-2)^{\frac{f}{2}}. 
\]
We simply write 
$\hat{\xi}^{2,4}$ for $\hat{\xi}^2 + \hat{\xi}^4$, 
and use similar notations also for other sums. 
We have 
\begin{equation}\label{evenphi}
 \phi_{\zeta'} ( 1+\hat{\xi}^{2,4}\delta_2 ^{-1}+ 
 \hat{\xi}^{1,2,3}\delta_2 ^{-2} ) =1
\end{equation}
for $\xi \in k$, since 
\begin{align*}
 \Nr_{E(\theta_2)/E} (1+\hat{\xi}\delta_4 \theta_2 ^{-1})
 &= \Nr_{E(\theta_2)/E} \left( 
 \theta_2 ^{-1}(\theta_2 +\hat{\xi}\delta_4 )\right) 
 = \Nr_{E(\delta_4)/E} \left( 
 \delta_4 ^{-2} ( \delta_4^2 -\hat{\xi} - \hat{\xi}^2 \delta_4 ) 
 \right) \\ 
 &=  \Nr_{E(\delta_4)/E} \left( 
 \delta_4 ^{-2} \left( (1-\hat{\xi}^2)\delta_4 +\delta_2 -\hat{\xi} 
 \right) \right) \\ 
 &\equiv 1+\hat{\xi}^{2,4}\delta_2 ^{-1}+ 
 \hat{\xi}^{1,2,3}\delta_2 ^{-2} 
 \pmod{1/2}. 
\end{align*}
Therefore we see that 
$\phi_{\zeta'} (1+\delta_2^{-1}x) =\pm \sqrt{-1}$ 
for $x \in \mathcal{O}_E$ if 
$\bar{x} \neq \xi^2 +\xi^4$ for any 
$\xi \in k$. 
Then we have 
\[
 \sum_{x \in \mathcal{O}_E / \mathfrak{p}_E} 
 \phi_{\zeta'} (1+\delta_2^{-1}x)^{-1} 
 \psi_E (\delta_2^2 x) = 
 \frac{1}{2} \sum_{\xi \in k} 
 \phi_{\zeta'} \bigl( 1+\hat{\xi}^{2,4}\delta_2^{-1} 
 \bigr)^{-1} 
 \psi_E \bigl( \delta_2^2 \hat{\xi}^{2,4} \bigr) , 
\] 
since we have already known that 
\[
 \sum_{x \in \mathcal{O}_E / \mathfrak{p}_E} 
 \phi_{\zeta'} (1+\delta_2^{-1}x)^{-1} 
 \psi_E (\delta_2^2 x) = \pm (-2)^{f/2} \in \mathbb{Q}. 
\]
We have 
\[ 
 \sum_{\xi \in k} 
 \phi_{\zeta'} \bigl( 1+\hat{\xi}^{2,4} \delta_2^{-1} 
 \bigr)^{-1} 
 \psi_E \bigl( \delta_2^2 \hat{\xi}^{2,4} \bigr) 
 = \sum_{\xi \in k} 
 \phi_{\zeta'} \bigl( 1+
 \hat{\xi}^3 \delta_2^{-2} \bigr) 
 \psi_E \bigl( \delta_2^2 \hat{\xi}^{2,4} \bigr), 
\]
since 
\[
 \phi_{\zeta'} ( 1+\hat{\xi}^{2,4} \delta_2^{-1} )^{-1} = 
 \phi_{\zeta'} ( 1+
 \hat{\xi}^{1,2,3} 
 \delta_2^{-2} ) = 
 \phi_{\zeta'} ( 1+
 \hat{\xi}^3 \delta_2^{-2} )
\]
by \eqref{eq:KerTr}, \eqref{eq:phichi} and \eqref{evenphi}. 
Further we have 
\begin{align*} 
 \sum_{\xi \in k} 
 \phi_{\zeta'} ( 1+ \hat{\xi}^3 \delta_2^{-2} )
 \psi_E ( \delta_2^2 \hat{\xi}^{2,4} ) 
 &= \sum_{\xi \in k} 
 \chi_2 \left( \Tr_{k/\mathbb{F}_2} (\xi^3 +\xi^2 +\xi^4) \right) \\ 
 &= \sum_{\xi \in k} 
 \chi_2 \left( \Tr_{k/\mathbb{F}_2} (\xi^3) \right) =-2(-2)^{\frac{f}{2}} 
\end{align*}
by Lemma \ref{phivarkappa} and \eqref{eq:KerTr}, 
because 
\[ 
 \lvert \{ (x,y) \in k^2 \mid x^2 +x =y^3 \} \rvert = 
 \lvert \mathcal{E} (\mathbb{F}_q ) \rvert -1 
 = q -2(-2)^{\frac{f}{2}} 
\]
by Lemma \ref{ratell}. 
Thus we have the claim in the case where $f$ is even. 

Next we consider the case where $f$ is odd. 
We define $n_f$ and $m_f$ as in Lemma \ref{reclaw}. 
We treat only the case where 
$n_f =m_f =1$. 
The other cases are proved similarly. 

We assume that $n_f =m_f =1$, 
which is equivalent to that $f \equiv 1 \bmod 8$. 
We put $\eta=\eta_2^f$. 
By Lemma \ref{reclaw} and the definition of 
$\phi_{\zeta'}$, we have 
$\phi_{\zeta'} (\delta_2) =q/\eta$. 
Hence, to prove \eqref{aim}, it suffices to show 
\[ 
 \sum_{x \in \mathcal{O}_E / \mathfrak{p}_E} 
 \phi_{\zeta'} (1+\delta_2^{-1}x)^{-1} 
 \psi_E (\delta_2^2 x) = -q^2 \eta^{-3}. 
\]
By \eqref{eq:phichi}, we have 
\begin{equation}\label{oddphi2'}
 \phi_{\zeta'} ( 
 1+ \hat{\xi}^{1,2} \delta_2 ^{-2} ) =1 
\end{equation} 
for $\xi \in k$. 
We have 
\begin{align}
 \phi_{\zeta'} \Bigl( 
 1+\Tr_{E_2/E} (\hat{\xi}^{2,4} ) \delta_2 ^{-1}+ 
 \bigl( \Tr_{E_2/E}(\hat{\xi}^{1,3} 
 ) + 
 \Nr_{E_2/E}( \hat{\xi}^{2,4} ) \bigr) \delta_2 ^{-2} \Bigr) =1 
 \label{oddphi1} 
\end{align}
for $\xi \in k_2$ by \eqref{oddphi2'} and 
\begin{align*}
 &\Nr_{E_2(\theta_2)/E} (1+\hat{\xi}\delta_4 \theta_2 ^{-1}) 
 = \Nr_{E_2(\delta_4)/E} \left( 
 \delta_4 ^{-2} ( \delta_4^2 -\hat{\xi} - \hat{\xi}^2 \delta_4 ) 
 \right) \\ 
 & \quad =  \Nr_{E_2(\delta_4)/E} \left( 
 \delta_4 ^{-2} \left( (1-\hat{\xi}^2)\delta_4 +\delta_2 -\hat{\xi} -\zeta_3 
 \right) \right) \\ 
 & \quad =  \Nr_{E_2/E} \left( 
 (\zeta_3-\delta_2)^{-2} 
 \left( (1-\hat{\xi}^2)^2(\zeta_3-\delta_2) 
 +(1-\hat{\xi}^2)(\delta_2 -\hat{\xi}-\zeta_3 ) 
 +(\delta_2 -\hat{\xi} -\zeta_3)^2 
 \right) \right) \\ 
 & \quad \equiv 1+\Tr_{E_2/E} ( \hat{\xi}^{2,4} ) \delta_2 ^{-1} 
 + \bigl( \Tr_{E_2/E}(\hat{\xi}^{1,3} +\hat{\xi}^2 \zeta_3^2 
 +\hat{\xi}^4 \zeta_3 ) + 
 \Nr_{E_2/E}(\hat{\xi}^{2,4}) \bigr) \delta_2 ^{-2} 
 \pmod{1/2}. 
\end{align*}
Then we have 
\begin{equation}\label{oddphi1'}
 \phi_{\zeta'} ( 
 1+ \hat{\xi}^{2,4} \delta_2 ^{-1} +
 \hat{\xi}^{1,3} \delta_2 ^{-2} ) =1 
\end{equation} 
for $\xi \in k$ 
by \eqref{oddphi2'} and \eqref{oddphi1}, 
because 
\begin{align*}
 \Tr_{k_2/k} (\xi^3) +\Nr_{k_2/k} (\xi^2 +\xi^4) =
 \Tr_{k_2/k} (\xi)^3 + 
 \bigl( \Nr_{k_2/k} &(\xi^2) + \xi^{q+1} \Tr_{k_2/k} (\xi) \bigr) \\ 
 &+ 
 \bigl( \Nr_{k_2/k} (\xi^2) + \xi^{q+1} \Tr_{k_2/k} (\xi) \bigr)^2 
\end{align*}
for $\xi \in k_2$. 
Since 
\[
 \Nr_{E_2(\theta_2)/E} (\theta_2/\delta_4) = 
 \Nr_{E_2(\delta_4)/E} (-\delta_4) = 
 \Nr_{E_2/E} (\zeta_3 -\delta_2) = 
 \delta_2 ^2 +\delta_2 +1 , 
\]
we have 
\[
 \phi_{\zeta'} (\delta_2 ^2 +\delta_2 +1) 
 =\phi' ( g(1,0,0),-2 )=(-2^{-1})^{-f}=-q. 
\]
Hence, we obtain 
$\phi_{\zeta'} (1+\delta_2 ^{-1} +\delta_2 ^{-2})=-\eta^2 q^{-1}$. 
Therefore we have 
\[
 \phi_{\zeta'} 
 \bigl( 1+ (\hat{\xi}^{2,4} +1) \delta_2^{-1} \bigr)^{-1} 
 \psi_E \bigl( \delta_2^2 (\hat{\xi}^{2,4} +1) \bigr) 
 = -\frac{q}{\eta^2} 
 \phi_{\zeta'} 
 (1+ \hat{\xi}^{2,4} \delta_2^{-1} )^{-1} 
 \psi_E (\delta_2^2 \hat{\xi}^{2,4} ) 
\]
by \eqref{oddphi2'} and 
$\phi_{\zeta'} (1+\delta_2^{-2}) = 
 \psi_E (\delta_2)=\psi_E (\delta_2^2)$, 
which follows from Lemma \ref{phivarkappa} 
and 
\[
 \Tr_{E/F} (\delta_2) =1 \equiv -1+2 (\zeta'' \varpi^{1/3})^{-1} = 
 \Tr_{E/F} (\delta_2^2) \pmod{2/3}. 
\]
Then we have 
\begin{align*}
 \sum_{x \in \mathcal{O}_E / \mathfrak{p}_E} 
 \phi_{\zeta'} (1+\delta_2^{-1}x)^{-1} 
 \psi_E (\delta_2^2 x) =
 \frac{1}{2} \left( 1-\frac{q}{\eta^2} \right) 
 \sum_{\xi \in k} 
 \phi_{\zeta'} 
 (1+\hat{\xi}^{2,4}\delta_2^{-1} )^{-1} 
 \psi_E (\delta_2^2 \hat{\xi}^{2,4}), 
\end{align*}
because 
$\xi \in \Ker \Tr_{k/\mathbb{F}_2} 
 =\{ {\xi'}^2 +{\xi'}^4 \mid \xi' \in k \}$ 
if and only if 
$\xi +1 \notin \Ker \Tr_{k/\mathbb{F}_2}$. 
Therefore, it suffices to show that 
\[
 \sum_{\xi \in k} 
 \phi_{\zeta'} 
 (1+\hat{\xi}^{2,4}\delta_2^{-1} )^{-1} 
 \psi_E (\delta_2^2 \hat{\xi}^{2,4}) = 
 -(-2)^{\frac{f+1}{2}}. 
\]
On the other hand, 
we have 
\begin{align*}
 \sum_{\xi \in k} 
 \phi_{\zeta'} 
 (1+\hat{\xi}^{2,4}\delta_2^{-1} )^{-1} 
 \psi_E (\delta_2^2 \hat{\xi}^{2,4}) &= 
 \sum_{\xi \in k} 
 \phi_{\zeta'} 
 (1+ \hat{\xi}^{1,3} \delta_2^{-2} ) 
 \psi_E (\delta_2^2 \hat{\xi}^{2,4}) \\ 
 &= \sum_{\xi \in k} 
 \chi_2 \left( \Tr_{k/\mathbb{F}_2} (\xi +\xi^3) \right) = 
 -(-2)^{\frac{f+1}{2}} 
\end{align*}
by \eqref{eq:KerTr}, \eqref{oddphi1'} and Lemma \ref{phivarkappa}, 
because 
\[ 
 \lvert \{ (x,y) \in k^2 \mid x^2 +x =y^3 +y \} \rvert = 
 \lvert \mathcal{E}'(\mathbb{F}_q ) \rvert -1 
 = q -(-2)^{\frac{f+1}{2}} 
\]
by Lemma \ref{ratell} under the assumption 
$f \equiv 1 \bmod 8$. 
Thus we have proved the claim. 
\end{proof}
\begin{thm}\label{expLLC}
For $\zeta' \in k^{\times}$, 
$\chi \in (k^{\times})^{\vee}$ and 
$c \in \overline{\mathbb{Q}}_{\ell}^{\times}$, 
we have 
$\mathrm{LL}_{\ell} (\tau_{\zeta', \chi, c}) = 
 \pi_{\zeta', \chi, c}$. 
\end{thm}
\begin{proof}
We may assume that 
$\chi=1$ and $c=1$ by character twists. 
By \cite[50.3 Lemma 1 and (50.3.2)]{BHLLCGL2} and 
Proposition \ref{indLLC}, 
it suffices to show that 
\begin{align*}
 &\Lambda_{\zeta'}|_{K^{\times}} \circ \mathrm{Art}_K ^{-1} 
 =(\det \tau_{\zeta'}) \otimes \lvert \cdot \rvert^{-1}, \\ 
 &\Lambda_{\zeta'} ^3 |_{U_{\mathfrak{J}}^1} = 
 \Lambda_{F, \zeta'} |_{U_{\mathfrak{J}}^1}, \\ 
 &\Lambda_{\zeta'} ^3 (x)= \epsilon_{F/K} ^{v_{L'} (x)} 
 \Lambda_{F, \zeta'} (x) \ \textrm{for} \ 
 x \in L^{\times}. 
\end{align*}
The first equality follows from 
Lemma \ref{dettau}, and 
the third equality follows from 
the definition of $\Lambda_{F, \zeta'}$. 

We are going to show the second equality. 
We put 
\begin{equation*}
 U'_{\mathfrak{J}} =\biggl\{ 
 \begin{pmatrix}
 a & b \\
 c & d
 \end{pmatrix} 
 \in \mathfrak{J} \ \bigg| \ 
 a \equiv d \equiv 1 ,\ 
 b \equiv 0 \bmod \mathfrak{p} \biggr\}, \quad 
 U''_{\mathfrak{J}} =\biggl\{ 
 \begin{pmatrix}
 1 & b \\
 0 & 1 
 \end{pmatrix} 
 \in \mathfrak{J} \biggr\}. 
\end{equation*}
Then we have 
$\Lambda_{\zeta'} ^3 |_{U'_{\mathfrak{J}}} = 
 \Lambda_{F, \zeta'} |_{U'_{\mathfrak{J}}}$ 
by the definition of 
$\Lambda_{\zeta'}$ and 
$\Lambda_{F, \zeta'}$, 
because 
$U'_{\mathfrak{J}} \subset U_{\mathfrak{J}_F}^2$. 
On the other hand, we have 
\begin{align*}
 \Lambda_{F, \zeta'} 
 \biggl( 
 \begin{pmatrix}
 1 & b \\
 0 & 1 
 \end{pmatrix} \biggr) 
 &= 
 \Lambda_{F, \zeta'} 
 \biggl( \varphi 
 \begin{pmatrix}
 1 & b \\
 0 & 1 
 \end{pmatrix} \varphi^{-1} \biggr) 
 = 
 \Lambda_{F, \zeta'} 
 \biggl( 
 \begin{pmatrix}
 1 & 0 \\
 b\varpi & 1 
 \end{pmatrix} \biggr) \\ 
 &= 
 \psi_F (\hat{\zeta}'^{-2} b) 
 = \psi_K (\hat{\zeta}'^{-2} b)^3 
 =
 \Lambda_{\zeta'} 
 \biggl( 
 \begin{pmatrix}
 1 & b \\
 0 & 1 
 \end{pmatrix} \biggr)^3 
\end{align*}
for $b \in \mathcal{O}_K$. 
Therefore we have the claim, 
because $U_{\mathfrak{J}}^1$ 
is generated by 
$U'_{\mathfrak{J}}$ and 
$U''_{\mathfrak{J}}$. 
\end{proof}

\noindent
Naoki Imai\\ 
Graduate School of Mathematical Sciences, 
The University of Tokyo, 3-8-1 Komaba, Meguro-ku, 
Tokyo 153-8914, Japan\\ 
naoki@ms.u-tokyo.ac.jp\\ 

\noindent
Takahiro Tsushima\\ 
Department of Mathematics and Informatics, 
Faculty of Science, Chiba University, 
1-33 Yayoi-cho, Inage, Chiba, 263-8522, Japan\\
tsushima@math.s.chiba-u.ac.jp


\begin{thebibliography}{JPSS81}
\providecommand{\url}[1]{\texttt{#1}}
\providecommand{\urlprefix}{URL }
\providecommand{\eprint}[2][]{\url{#2}}

\bibitem[BH06]{BHLLCGL2}
C.~J. Bushnell and G.~Henniart, The local {L}anglands conjecture for {$\rm
  GL(2)$}, vol. 335 of Grundlehren der Mathematischen Wissenschaften,
  Springer-Verlag, Berlin, 2006.

\bibitem[Car90]{CarNaLT}
H.~Carayol, Non-abelian {L}ubin-{T}ate theory, in Automorphic forms, {S}himura
  varieties, and {$L$}-functions, {V}ol. {II} ({A}nn {A}rbor, {MI}, 1988),
  vol.~11 of Perspect. Math., pp. 15--39, Academic Press, Boston, MA, 1990.

\bibitem[Dat07]{DatLTel}
J.-F. Dat, Th\'{e}orie de {L}ubin-{T}ate non-ab\'{e}lienne et
  repr\'{e}sentations elliptiques, Invent. Math. 169 (2007), no.~1, 75--152.

\bibitem[Del84]{Delp0}
P.~Deligne, Les corps locaux de caract\'eristique {$p$}, limites de corps
  locaux de caract\'eristique {$0$}, in Representations of reductive groups
  over a local field, Travaux en Cours, pp. 119--157, Hermann, Paris, 1984.

\bibitem[Far04]{FarCohpdiv}
L.~Fargues, Cohomologie des espaces de modules de groupes {$p$}-divisibles et
  correspondances de {L}anglands locales, Ast\'erisque  (2004), no. 291,
  1--199, {V}ari\'et\'es de Shimura, espaces de Rapoport-Zink et
  correspondances de Langlands locales.

\bibitem[HG94]{HGEvecLT}
M.~J. Hopkins and B.~H. Gross, Equivariant vector bundles on the {L}ubin-{T}ate
  moduli space, in Topology and representation theory ({E}vanston, {IL}, 1992),
  vol. 158 of Contemp. Math., pp. 23--88, Amer. Math. Soc., Providence, RI,
  1994.

\bibitem[Hub98]{HubCompl}
R.~Huber, A comparison theorem for {$l$}-adic cohomology, Compositio Math. 112
  (1998), no.~2, 217--235.

\bibitem[IT15a]{ITCohrigc}
N.~Imai and T.~Tsushima, Cohomology of rigid curves with semi-stable coverings,
  Asian J. Math. 19 (2015), no.~4, 637--649.

\bibitem[IT15b]{ITlgsw1}
N.~Imai and T.~Tsushima, Local {G}alois representations of {S}wan conductor
  one, 2015, arXiv:1509.02960.

\bibitem[IT17]{ITStab3}
N.~Imai and T.~Tsushima, Stable models of {L}ubin-{T}ate curves with level
  three, Nagoya Math. J. 225 (2017), 100--151.

\bibitem[IT18]{ITsimpJL}
N.~Imai and T.~Tsushima, Local {J}acquet--{L}anglands correspondences for
  simple supercuspidal representations, Kyoto J. Math. 58 (2018), no.~3,
  623--638.

\bibitem[IT20]{ITsimptame}
N.~Imai and T.~Tsushima, Affinoids in the {L}ubin-{T}ate perfectoid space and
  simple supercuspidal representations {I}: tame case, Int. Math. Res. Not.
  IMRN  (2020), no.~22, 8251--8291.

\bibitem[IT21]{ITsimpwild}
N.~Imai and T.~Tsushima, Affinoids in the {L}ubin-{T}ate perfectoid space and
  simple supercuspidal representations {II}: wild case, Math. Ann. 380 (2021),
  no. 1-2, 751--788.

\bibitem[JPSS81]{JPSSCond}
H.~Jacquet, I.~I. Piatetski-Shapiro and J.~Shalika, Conducteur des
  repr\'{e}sentations du groupe lin\'{e}aire, Math. Ann. 256 (1981), no.~2,
  199--214.

\bibitem[Kut80]{KutLcGl2}
P.~Kutzko, The {L}anglands conjecture for {${\rm Gl}_{2}$} of a local field,
  Ann. of Math. (2) 112 (1980), no.~2, 381--412.

\bibitem[Mie14]{MieGeomJL}
Y.~Mieda, Geometric approach to the local {J}acquet-{L}anglands correspondence,
  Amer. J. Math. 136 (2014), no.~4, 1067--1091.

\bibitem[Ser68]{SerCL}
J.-P. Serre, Corps locaux, Hermann, Paris, 1968, deuxi{\`e}me {\'e}dition,
  Publications de l'Universit{\'e} de Nancago, No. VIII.

\bibitem[Str05]{StrJLLT}
M.~Strauch, On the {J}acquet-{L}anglands correspondence in the cohomology of
  the {L}ubin-{T}ate deformation tower, Ast\'{e}risque  (2005), no. 298,
  391--410, {F}ormes {A}utomorphes (I).

\bibitem[Tsu16]{TsuLTodd}
T.~Tsushima, On non-abelian Lubin-Tate theory for GL(2) in the odd equal
  characteristic case, 2016, arXiv:1604.08857.

\end{thebibliography}
\end{document}